\documentclass[11pt,reqno]{amsart}
\usepackage{amsmath, amsfonts, amsthm, amssymb}
\newtheorem{Theorem}{Theorem}[section]
\newtheorem{Lemma}{Lemma}[section]

\theoremstyle{definition}

\theoremstyle{remark}
\newtheorem{Remark}{Remark}[section]

\numberwithin{equation}{section}

\renewcommand{\r}{\rho}
\renewcommand{\t}{\theta}

\renewcommand{\u}{{\bf u}}

\newcommand{\R}{{\mathbb R}}
\newcommand{\Dv}{{\rm div}}
\newcommand{\tr}{{\rm tr}}

\newcommand{\dl}{\delta}

\def\f{\frac}
\renewcommand{\O}{\Omega}
\def\ov{\overline}

\def\D{\Delta }

\def\hf1{^\f{1}{1-\xi^2}}

\def\S{{\mathcal S}}
\def\T{{\mathcal T}}

\def\be{\begin{equation}}
\def\en{\end{equation}}
\def\bs{\begin{split}}
\def\es{\end{split}}

\newcommand{\F}{{\mathtt F}}

\begin{document}

\author{Xianpeng HU and Dehua Wang}
\address{Department of Mathematics, University of Pittsburgh,
                           Pittsburgh, PA 15260, USA.}
\email{xih15@pitt.edu}
\address{Department of Mathematics, University of Pittsburgh,
                           Pittsburgh, PA 15260, USA.}
\email{dwang@math.pitt.edu}

\title[Strong Solutions to the Compressible Viscoelastic Fluids]
{Strong Solutions to the Three-Dimensional Compressible Viscoelastic Fluids}

\date{\today}

\begin{abstract}
The existence and uniqueness of the local strong solution to the
three-dimensional compressible viscoelastic fluids near the
equilibrium is established. In addition to the uniform estimates on the velocity,
some essential uniform estimates  on the
density and the deformation gradient   are also obtained.
\end{abstract}

\maketitle

\section{Introduction}
Viscoelastic fluids exhibit a combination of both fluid and solid characteristics,
and keep memory of their past deformations.  The interaction between the
microscopic elastic properties and the macroscopic fluid motions
leads to the rich and complicated rheological phenomena in
viscoelastic fluids, and also causes formidable analytic and
numerical challenges in mathematical analysis.
We consider the following equations of three-dimensional
compressible flow of viscoelastic fluids \cite{CD, Joseph, LZ}:
\begin{subequations} \label{e1e}
\begin{align}
&\r_t +\Dv(\r\u)=0,\label{e1e1}\\
&(\r\u)_t+\Dv\left(\r\u\otimes\u\right)-\mu\D \u-(\lambda+\mu)\nabla\Dv\u+\nabla P(\r)=\Dv(\r\, \F\,\F^\top),\label{e1e2}\\
&\F_t+\u\cdot\nabla\F=\nabla\u \, \F,\label{e1e3}
\end{align}
\end{subequations}
where $\r$ stands for the density, $\u\in \R^3$ the velocity, and
$\F\in M^{3\times 3}$ (the set of  $3\times 3$ matrices)  the
deformation gradient. The viscosity coefficients $\mu, \lambda$
are two constants satisfying $\mu>0, 2\mu+3\lambda>0$, which
ensures that the operator $-\mu\D \u-(\lambda+\mu)\nabla\Dv\u$ is
a strongly elliptic operator. The pressure term $P(\r)$ is an
increasing and convex function of $\rho$ for $\rho>0$, and in
particular, we take $P(\r)=\r^\gamma$ with $\gamma>1$ a conatant.
The symbol $\otimes$ denotes the Kronecker tensor product,
$\F^\top$ means the transpose matrix of $\F$, and the notation
$\u\cdot\nabla\F$ is understood to be $(\u\cdot\nabla)\F$. As
usual, we call equation \eqref{e1e1} the continuity equation. For
system \eqref{e1e}, the corresponding elastic energy is chosen to
be  the special form of the Hookean linear elasticity:
$$W(\F)=\frac12|\F|^2,$$
which, however, does not reduce the essential difficulties for analysis.
The methods and results of this paper can be applied to more general cases.

In this paper, we consider equations \eqref{e1e} subject to the
initial condition:
\begin{equation}\label{IC1}
(\r, \u, \F)|_{t=0}=(\r_0(x), \u_0(x), \F_0(x)), \quad x\in\R^3,
\end{equation}
and we are  interested  in the existence and uniqueness of strong
solution to the initial-value problem \eqref{e1e}-\eqref{IC1} near
its equilibrium state in the three dimensional space $\R^3$. Here
the {equilibrium state} of the system \eqref{e1e} is defined as:
$\r$ is a positive constant (for simplicity, $\r=1$), $\u=0$, and
$\F=I$ (the identity matrix in $M^{3\times 3}$). We introduce a
new unknown variable $E$ by setting $$\F=I+E.$$   Then,
\eqref{e1e} becomes
\begin{subequations} \label{e1}
\begin{align}
&\r_t +\Dv(\r\u)=0,\label{e11}\\
&(\r\u)_t+\Dv\left(\r\u\otimes\u\right)-\mu\D \u-(\mu+\lambda)\nabla\Dv\u+\nabla P(\r)=\Dv(\r (I+E)(I+E)^\top),\label{e12}\\
&E_t+\u\cdot\nabla E=\nabla\u E+\nabla\u,\label{e13}
\end{align}
\end{subequations}
with the initial data
\begin{equation}\label{IC}
(\r, \u, E)|_{t=0}=(\r_0(x), \u_0(x), E_0(x)), \quad x\in\R^3.
\end{equation}
By a {\em strong solution} to \eqref{e1}-\eqref{IC}, we mean a
triplet $(\r, \u, E)$ satisfying \eqref{e1} almost everywhere with
initial condition \eqref{IC}, in particular, $\u(\cdot, t)\in W^{2,q}$ and $(\r(t,\cdot),E(t,\cdot))\in W^{1,q}$
with $q\in(3,\infty)$ in some time interval $[0,T]$ for $T>0$ in
this paper.

When the density $\r$ is a constant, system \eqref{e1e} governs
the homogeneous incompressible viscoelastic fluids, and there
exist rich results in the literature for the global existence of
classical solutions (namely in $H^3$ or other functional spaces
with much higher regularity); see \cite{CM, CZ, KP, LLZH2, LLZH,
LZ, LLZ, LZP, LW} and the references therein. When the density
$\r$ is not a constant, the question related to existence becomes
much more complicated and not much has been done. In \cite{LLZH31}
the authors considered the global existence of classical solutions
in $H^3$ of small perturbation near its equilibrium for the
compressible viscoelastic fluids without the pressure term. One of
the main difficulties in proving the global existence is the
lacking of the dissipative estimate for the deformation gradient
and the gradient of the density. To overcome this difficulty, for
incompressible cases, the authors in \cite{LLZH} introduced an
auxiliary function to obtain the dissipative estimate, while the
authors in \cite{LZ} directly deal with the quantities such as
$\Delta \u+\Dv \F$. Those methods can provide them with some good
estimates, partly because of their high regularity of $(\u,\F)$.
However, in this paper, we deal with the strong solution with much
less regularity in $W^{2,q}$, $q\in(3,\infty]$, hence those
methods do not apply. For that purpose, we need a new method to
overcome this obstacle, and we find that
 a combination between the velocity and the
convolution of the divergence of the deformation gradient with the
fundamental solution of Laplace operator will develop some good
dissipative estimates which may be very useful for the global
existence. The local existence is established using a fixed point
theorem. Some uniform estimates on the solution are also obtained.
These estimates are essential for the global existence although one of the estimates needs to be improved
in order to establish the global existence.

The viscoelastic fluid system \eqref{e1e} can be regarded as a
combination of compressible Navier-Stokes equations with the
source term $\Dv(\r\F\F^\top)$ and the equation \eqref{e1e3}. For
the global existence of classical solutions with
 small perturbation near an equilibrium for the compressible
Navier-Stokes equations, we refer the   reader to
\cite{MT1, MT, AI, SS} and the references cited therein. We
remark that,  for  the  nonlinear compressible inviscid elastic systems,
the existence of solutions  was established by Sideris-Thomases in \cite{ST} under the null condition; see also \cite{ST2} for a related discussion.

The existence of global weak solutions with large initial data of
\eqref{e1e} is still an outstanding open question. In this
direction for the homogeneous incompressible viscoelastic fluids,
when the contribution of the strain rate (symmetric part of
$\nabla\u$) in the constitutive equation is neglected,
Lions-Masmoudi in \cite{LM} proved the global existence of weak
solutions with large initial data for the Oldroyd model. Also
Lin-Liu-Zhang in \cite{LLZ} proved the existence of global weak
solutions with large initial data for the incompressible
viscoelastic fluids when the velocity satisfies the Lipschitz
condition. When dealing with the global existence of weak
solutions of the viscoelastic fluid system \eqref{e1e} with large
data, the rapid oscillation of the density and the
non-compatibility between the quadratic form and the weak
convergence are  two of the major difficulties.

The rest of the paper is organized as follows. In Section 2, we
recall briefly the compressible viscoelastic fluids from some
basic mechanics and  conservation laws. In Section 3, we state our
main results, including the local existence and uniqueness of the
strong solution to the system \eqref{e1}-\eqref{IC}, as well as
uniform estimates which may be very useful for the proof of the
global existence. In Section 4, we prove the local existence via a
fixed-point theorem. In Section 5, we prove the uniqueness of the
solution obtained in Section 4. In Section 6, we establish some
uniform {\it a priori} estimates, especially on the dissipation of
the deformation gradient and gradient of the density. In Section
7, we improve the uniform estimates in Section 6.

\section{Background of Mechanics for Viscoelastic Fluids}

To provide a better understanding of system \eqref{e1e}, we recall briefly some background of viscoelastic fluids from mechanics in this section.

First, we discuss  the deformation gradient $\F$. The dynamics  of
a velocity field $\u(x,t)$ in mechanics can be described by the
flow map or particle trajectory $x(t,X)$, which is a time
dependent family of orientation preserving diffeomorphisms defined
by:
\begin{equation}\label{x2}
\f{d}{dt}x(t,X)=\u(t, x(t,X)),\quad
x(0,X)=X,
\end{equation}
where the material point  $X$  (Lagrangian coordinate) is deformed to the spatial position $x(t,X)$, the reference (Eulerian) coordinate  at time $t$.
The deformation gradient $\widetilde{\F}$ is  defined as
$$\widetilde{\F}(t,X)=\f{\partial x}{\partial X}(t,X),$$
which describes the change of configuration, amplification or pattern during
the dynamical process,
and satisfies the following equation by changing the order of  differentiation:
\begin{equation}\label{x1}
\f{\partial\widetilde{\F}(t,X)}{\partial t}=\f{\partial\u(t,x(t,X))}{\partial X}.
\end{equation}
In the Eulerian coordinate, the corresponding deformation gradient $\F(t,x)$ is defined as
$$\F(t, x(t,X))=\widetilde{\F}(t,X).$$
Equation \eqref{x1}, combined with the chain rule and \eqref{x2}, gives
\begin{equation*}
\begin{split}
\partial_t\F(t, x(t,X))+\u\cdot\nabla\F(t,x(t,X))&=\partial_t\F(t, x(t,X))+\f{\partial\F(t,x(t,X))}{\partial x}
\cdot\f{\partial x(t,X)}{\partial t}\\
&=\f{\partial\widetilde{\F}(t,X)}{\partial t}=\f{\partial\u(t,x(t,X))}{\partial X}
   =\f{\partial\u(t, x(t,X))}{\partial x}\f{\partial x}{\partial X}\\
&=\f{\partial\u(t, x(t,X))}{\partial x}\widetilde{\F}(t, X)=\nabla\u\cdot\F,
\end{split}
\end{equation*}
which is exactly  equation \eqref{e1e3}. Here, and in what follows, we use the conventional notations:
$$(\nabla\u)_{ij}=\f{\partial u_i}{\partial x_j},\quad(\nabla\u\,\F)_{i,j}=(\nabla\u)_{ik}\F_{kj},\quad (\u\cdot\nabla\F)_{ij}
=u_k\f{\partial \F_{ij}}{\partial x_{k}},$$
and summation over repeated indices will always be well understood. In viscoelastic fluids, \eqref{e1e3} can also be interpreted
as the consistency of the flow maps generated by the velocity field $\u$ and the deformation gradient $\F$.

The difference between fluids and solids lies in the fact that,  in fluids, such as Navier-Stokes equations \cite{AI},
the internal energy can be determined solely by the determinant
part of $\F$ (equivalently the density $\r$, and hence, \eqref{e1e3} can be disregarded); while in elasticity, the energy depends on all information of $\F$.

In the continuum physics, if we assume that the material is
homogeneous, then the conservation laws of mass and of momentum
become \cite{CD, LLZH, RHN, ST2}:
\begin{equation}\label{x3}
\partial_t\r+\Dv(\r\u)=0,
\end{equation}
and
\begin{equation}\label{x4}
\partial_t(\r\u)+\Dv(\r\u\otimes\u)-\mu\Delta\u-(\mu+\lambda)\nabla\Dv\u+\nabla P(\r)=\Dv((\det\F)^{-1}S\F^\top),
\end{equation}
where
\begin{equation}\label{x5}
\r\det\F=1,
\end{equation}
and
\begin{equation}\label{y1}
S_{ij}(\F)=\f{\partial W}{\partial \F_{ij}}.
\end{equation}
Here $S$, $\r S\F^\top$, $W(\F)$ denote $\textit{Piola-Kirchhoff stress}$, $\textit{Cauchy stress}$, and the elastic energy of
the material, respectively. Recall that the condition \eqref{y1} implies that the material is  hyperelastic \cite{LW}.
In the case of Hookean (linear) elasticity \cite{LLZH2, LLZH, LZP},
\begin{equation}\label{x6}
W(\F)=\f{1}{2}|\F|^2=\f{1}{2}tr(\F\F^\top),
\end{equation}
where the notation ``$\tr$"  stands for the trace operator of a matrix,
and hence,
\begin{equation}\label{x7}
S(\F)=\F.
\end{equation}
Combining equations \eqref{x2}-\eqref{x7} together, we obtain  system \eqref{e1e}.

If the viscoelastic system \eqref{e1e} satisfies
$$\Dv(\r_0\F_0^\top)=0,$$
initially at $t=0$ (with $\F_0=I+E_0$), 
it was verified in 
\cite{LZ} (see Proposition 3.1)  that this condition will insist in time, that is,
\begin{equation}\label{cc1}
\Dv(\r(t)\F(t)^\top)=0,\quad\textrm{for}\quad t\ge 0.
\end{equation}

Another hidden, but important, property of the viscoelastic fluids
system \eqref{e1e} is concerned with the curl of the deformation
gradient (for the incompressible case, see \cite{LLZH2, LLZH}).
Actually, the following lemma says that the curl of the
deformation gradient is of higher order.

\begin{Lemma}\label{curl}
Assume that \eqref{e1e3} is satisfied and $(\u, \F)$ is the
solution of the system \eqref{e1e}. Then the following identity
\begin{equation}\label{curl1}
\F_{lk}\nabla_l \F_{ij}=\F_{lj}\nabla_l \F_{ik}
\end{equation}
holds for all time $t> 0$ if it initially satisfies \eqref{curl1}.
\end{Lemma}

Again, throughout this paper, the standard summation notation over the
repeated index is always adopted.
\begin{proof}
First, we establish the evolution equation for the equality
$\F_{lk}\nabla_l \F_{ij}-\F_{lj}\nabla_l \F_{ik}$. Indeed, by the
equation \eqref{e1e3}, we can get
\begin{equation*}
\partial_t\nabla_l\F_{ij}+\u\cdot\nabla\nabla_l
\F_{ij}+\nabla_l\u\cdot\nabla\F_{ij}=\nabla_m\u_i\nabla_l\F_{mj}+\nabla_l\nabla_m\u_i
\F_{mj}.
\end{equation*}
Thus,
\begin{equation}\label{curl2}
\F_{lk}(\partial_t\nabla_l\F_{ij}+\u\cdot\nabla\nabla_l
\F_{ij})+\F_{lk}\nabla_l\u\cdot\nabla\F_{ij}=\F_{lk}\nabla_m\u_i\nabla_l\F_{mj}+\F_{lk}\nabla_l\nabla_m\u_i
\F_{mj}.
\end{equation}
Also, from \eqref{e1e3}, we obtain
\begin{equation}\label{curl3}
\nabla_l\F_{ij}(\partial_t\F_{lk}+\u\cdot\nabla\F_{lk})=\nabla_l\F_{ij}\nabla_m\u_l\F_{mk}.
\end{equation}

Now, adding \eqref{curl2} and \eqref{curl3}, we deduce that
\begin{equation}\label{curl4}
\begin{split}
\partial_t(\F_{lk}\nabla_l \F_{ij})+\u\cdot\nabla(\F_{lk}\nabla_l
\F_{ij})&=-\F_{lk}\nabla_l\u\cdot\nabla\F_{ij}+\F_{lk}\nabla_m\u_i\nabla_l\F_{mj}\\&\quad+\F_{lk}\nabla_l\nabla_m\u_i
\F_{mj}+\nabla_l\F_{ij}\nabla_m\u_l\F_{mk}\\
&=\F_{lk}\nabla_m\u_i\nabla_l\F_{mj}+\F_{lk}\nabla_l\nabla_m\u_i
\F_{mj}.
\end{split}
\end{equation}
Here, we used the identity which is derived by interchanging the
roles of indices $l$ and $m$:
$$\F_{lk}\nabla_l\u\cdot\nabla\F_{ij}=\F_{lk}\nabla_l\u_m\nabla_m\F_{ij}=\nabla_l\F_{ij}\nabla_m\u_l\F_{mk}.$$
Similarly, one has
\begin{equation}\label{curl5}
\begin{split}
\partial_t(\F_{lj}\nabla_l \F_{ik})+\u\cdot\nabla(\F_{lj}\nabla_l
\F_{ik})=\F_{lj}\nabla_m\u_i\nabla_l\F_{mk}+\F_{lj}\nabla_l\nabla_m\u_i
\F_{mk}.
\end{split}
\end{equation}
Subtracting \eqref{curl5} from \eqref{curl4} yields
\begin{equation}\label{curl6}
\begin{split}
&\partial_t(\F_{lk}\nabla_l \F_{ij}-\F_{lj}\nabla_l
\F_{ik})+\u\cdot\nabla(\F_{lk}\nabla_l \F_{ij}-\F_{lj}\nabla_l
\F_{ik})\\&\quad=\nabla_m\u_i(\F_{lk}\nabla_l\F_{mj}-\F_{lj}\nabla_l\F_{mk})+\nabla_l\nabla_m\u_i
(\F_{mj}\F_{lk}-\F_{mk}\F_{lj}).
\end{split}
\end{equation}
Due to the fact
$$\nabla_l\nabla_m\u_i=\nabla_m\nabla_l\u_i$$ in the sense of distributions, we
have, again by interchanging the roles of indices $l$ and $m$,
\begin{equation*}
\begin{split}
\nabla_l\nabla_m\u_i
(\F_{mj}\F_{lk}-\F_{mk}\F_{lj})&=\nabla_l\nabla_m\u_i
\F_{mj}\F_{lk}-\nabla_l\nabla_m\u_i
\F_{mk}\F_{lj}\\
&=\nabla_l\nabla_m\u_i \F_{mj}\F_{lk}-\nabla_m\nabla_l\u_i
\F_{lk}\F_{mj}\\
&=(\nabla_l\nabla_m\u_i-\nabla_m\nabla_l\u_i)
\F_{lk}\F_{mj}=0.
\end{split}
\end{equation*}
From this identity, equation \eqref{curl6} can be simplified as
\begin{equation}\label{curl7}
\begin{split}
&\partial_t(\F_{lk}\nabla_l \F_{ij}-\F_{lj}\nabla_l
\F_{ik})+\u\cdot\nabla(\F_{lk}\nabla_l \F_{ij}-\F_{lj}\nabla_l
\F_{ik})\\&\quad=\nabla_m\u_i(\F_{lk}\nabla_l\F_{mj}-\F_{lj}\nabla_l\F_{mk}).
\end{split}
\end{equation}
Multiplying \eqref{curl7} by $\F_{lk}\nabla_l
\F_{ij}-\F_{lj}\nabla_l \F_{ik}$, we get
\begin{equation}\label{curl8}
\begin{split}
&\partial_t|\F_{lk}\nabla_l \F_{ij}-\F_{lj}\nabla_l
\F_{ik}|^2+\u\cdot\nabla|\F_{lk}\nabla_l \F_{ij}-\F_{lj}\nabla_l
\F_{ik}|^2\\&\quad=2(\F_{lk}\nabla_l \F_{ij}-\F_{lj}\nabla_l
\F_{ik})\nabla_m\u_i(\F_{lk}\nabla_l\F_{mj}-\F_{lj}\nabla_l\F_{mk})\\
&\quad\le 2\|\nabla\u\|_{L^\infty(\R^3)}\mathcal{M}^2,
\end{split}
\end{equation}
where $\mathcal{M}$ is defined as
$$\mathcal{M}=\max_{i,j,k}\{|\F_{lk}\nabla_l
\F_{ij}-\F_{lj}\nabla_l \F_{ik}|^2\}.$$ Hence, \eqref{curl8}
implies
\begin{equation}\label{curl9}
\partial_t\mathcal{M}+\u\cdot\nabla\mathcal{M}\le
2\|\nabla\u\|_{L^\infty(\R^3)}\mathcal{M}.
\end{equation}

On the other hand, the characteristics of $\partial_t
f+\u\cdot\nabla f=0$ is given by
$$\f{d}{ds}X(s)=\u(s,X(s)),\quad X(t)=x.$$
Hence, \eqref{curl8} can be rewritten as
\begin{equation}\label{curl10}
\f{\partial U}{\partial t}\le B(t,y)U,\quad
U(0,y)=\mathcal{M}_0(y),
\end{equation}
where
$$U(t,y)=\mathcal{M}(t, X(t,x)),\quad
B(t,y)=2\|\nabla\u\|_{L^\infty(\R^3)}(t, X(t,y)).$$
The differential inequality \eqref{curl10} implies that
$$U(t,y)\le U(0)\exp\left(\int_0^t B(s,y)ds\right).$$ Hence,
$$\mathcal{M}(t,x)\le \mathcal{M}(0)\exp\left(\int_0^t
2\|\nabla\u\|_{L^\infty(\R^3)}(s)ds\right).$$ Hence, if
$\mathcal{M}(0)=0$, then $\mathcal{M}(t)=0$ for all $t>0$, and
 the proof of the lemma is complete.
\end{proof}

\begin{Remark}
Lemma \ref{curl} can be interpreted from the physical viewpoint as
follows: formally, the fact that the Lagrangian derivatives
commute and the definition of the deformation gradient imply
$$\partial_{X_k}\widetilde{\F}_{ij}=\f{\partial^2x_i}{\partial X_k\partial X_j}=\f{\partial^2x_i}{\partial X_j\partial X_k}
=\partial_{X_j}\widetilde{\F}_{ik},$$
which is equivalent to, in the Eulerian coordinates,
$$\widetilde{\F}_{lk}\nabla_{l}\F_{ij}(t,x(t,X))=\widetilde{\F}_{lj}\nabla_{l}\F_{ik}(t, x(t,X)), $$
that is,
$$\F_{lk}\nabla_{l}\F_{ij}(t,x)=\F_{lj}\nabla_{l}\F_{ik}(t,x).$$
\end{Remark}

Using $\F=I+E$, \eqref{curl1} means
\begin{equation}\label{11111b}
\nabla_{k}E_{ij}+E_{lk}\nabla_{l}E_{ij}=\nabla_{j}E_{ik}+E_{lj}\nabla_{l}E_{ik}.
\end{equation}
According to \eqref{11111b}, it is natural to assume that the initial condition of $E$ in the viscoelastic fluids system \eqref{e1} should
satisfy the compatibility condition
\begin{equation}\label{11113}
\nabla_{k}E(0)_{ij}+E(0)_{lk}\nabla_{l}E(0)_{ij}=\nabla_{j}E(0)_{ik}+E(0)_{lj}\nabla_{l}E(0)_{ik}.
\end{equation}

Finally, if the density $\r$ is a constant,  then we have the following equations of incompressible
viscoelastic fluids (see \cite{CZ, LLZH2, LLZH, LZ, LLZ, LZP} and references
therein):
\begin{equation}\label{ICV}
\begin{cases}
&\Dv \u=0,\\
&\partial_t\u+\u\cdot\nabla\u-\mu\Delta\u+\nabla P=\Dv(\F\F^\top),\\
&\partial_t\F+\u\cdot\nabla\F=\nabla\u\,\F.
\end{cases}
\end{equation}

For more discussions on viscoelastic fluids  and related models, see \cite{BAH,CM,CD,Gurtin,Joseph,Larson,LLZ, LM,LW,RHN,Slemrod} and the references cited therein.

\section{Main Results}

In this Section, we state our main results.
The standard notations for Sobolev spaces $W^{s, q}$ and Besov spaces $B^s_{pq}$ (\cite{BL}) will be used.
Throughout this paper, the real interpolation method (\cite{BL}) will be adopted
and the following interpolation spaces will be needed
$$\left(L^q(\R^3), W^{2,q}(\R^3)\right)_{1-\f{1}{p},p}=B^{2(1-\f{1}{p})}_{qp},\quad
\left(L^q(\R^3), W^{1,q}(\R^3)\right)_{1-\f{1}{p},p}=B^{1-\f{1}{p}}_{qp}, \quad p, q\ge 1.$$
Now we introduce the following functional spaces to which the solution and initial conditions of the system \eqref{e1} will belong. Given $1\le p, q\le\infty$ and $T>0$, we set $Q_T=\R^3\times(0,T)$, and
$$\mathcal{W}^{p,q}(0,T):=\left\{\u: \u\in W^{1,p}(0,T; (L^q(\R^3))^3)\cap L^p(0,T;(W^{2,q}(\R^3))^3)\right\}$$
with the norm
$$\|\u\|_{\mathcal{W}^{p,q}(0,T)}:=\|\u\|_{W^{1,p}(0,T; L^q(\R^3))}+\|\u\|_{L^p(0,T; W^{2,q}(\R^3))},$$
as well as
$$V_0^{p,q}:=\left(B^{2(1-\f{1}{p})}_{qp}\cap B^{1-\f{1}{p}}_{qp}\right)^3\times \left(W^{1,q}(\R^3)\right)^{10}$$
with the norm
$$\|(f,g)\|_{V_0^{p,q}}:=\|f\|_{B^{2(1-\f{1}{p})}_{qp}}+\|f\|_{B^{1-\f{1}{p}}_{qp}}+\|g\|_{W^{1,q}(\R^3)}.$$
We denote
$$\mathcal{W}(0,T)=\mathcal{W}^{p,q}(0,T)\cap\mathcal{W}^{2,2}(0,T),$$
and
$$V_0=V_0^{p,q}\cap V_0^{2,2}.$$

Our first result is  the following local existence and uniqueness:

\begin{Theorem} \label{T20}
Let $T_0>0$ be given and  $(\u_0, \r_0, E_0)$  $\in  V_0$ with
$p\in [2,\infty), q\in(3,\infty)$. There exists a positive
constant $\dl_0<1$, depending on $T_0$,  $\mu$, and $\lambda$, such that if
\begin{equation}\label{xx1}
\|(\u_0, \r_0-1, E_0)\|_{V_0}\le \dl_0,
\end{equation}
then the initial-value problem \eqref{e1}-\eqref{IC}  has a unique  strong solution on $\R^3\times (0,T_0)$, satisfying
$\u\in \mathcal{W}(0,T_0)$ and
$$(\r-1, E)\in\left(W^{1,p}(0,T_0; L^q(\R^3)\cap L^2(\R^3))\cap L^p(0,T_0; W^{1,q}(\R^3)
\cap W^{1,2}(\R^3))\right)^{10}.$$
\end{Theorem}

\medskip

The solutions in Theorem \ref{T20} is local in time since
$\delta_0=\delta_0(T_0)$ implies that $T_0$ is finite for a given
$\delta_0\ll 1$.

\begin{Remark}\label{yy1}
Notice that if $q>3$, then by Theorem 5.15 in \cite{AD}, the
imbedding $W^{1,q}(\R^3)\hookrightarrow C^0_B(\R^3)$ is
continuous. Here, the notation $C_B^0(\R^3)$ means the spaces of
bounded, continuous functions in $\R^3$. Hence the condition
\eqref{xx1} implies that, if we choose $\dl_0$ sufficiently small,
by Sobolev's imbedding theorem, there exists a positive constant
$C_0$ such that
\begin{equation}\label{positive}
\r_0\ge C_0>0, \quad\textrm{for a.e.}\quad x\in\R^3.
\end{equation}
\end{Remark}

\begin{Remark}
An interesting case is the case $q\le p$. Indeed, by the real
interpolation method and Theorem 6.4.4 in \cite{BL}, we have
$$W^{2(1-\f{1}{p}),q}\subset B^{2(1-\f{1}{p})}_{qp},\quad W^{1-\f{1}{p},q}\subset B^{1-\f{1}{p}}_{qp}.$$
Then, if we replace the functional space $V_0^{p,q}$ in Theorem \ref{T20} by
$$\mathcal{V}_0^{p,q}:=\left((W^{2(1-\f{1}{p}),q}(\R^3))^3\cap (W^{1-\f{1}{p}, q}(\R^3))^3\right)\times
(W^{1,q}(\R^3))^{10},$$
Theorem \ref{T20} is still valid.
\end{Remark}

For the solutions claimed in Theorem \ref{T20}, the following
estimates hold uniformly in time.

\begin{Theorem}\label{T1}
Let $\dl_0$ be the same as  in Theorem \ref{T20} and assume
$0<\dl\ll \min\{\f{1}{3},\dl_0\}$.
\smallskip

{\rm (I)} If the initial data satisfies
$\|(\u_0, \r_0-1, E_0)\|_{V_0}\le \dl^2,$ and  $\mu$ and $\lambda$ satisfy the assumption \eqref{AA} below,
then the solution   $(\r, \u, E)$ constructed in Theorem \ref{T20} satisfies
 $$\|\nabla\r\|_{L^\infty(0,T_0; L^q(\R^3))}\le C\sqrt{\dl_0}, \quad
 \|\nabla E\|_{L^p(0,T_0; L^q(\R^3))}\le C\sqrt{\dl_0}.$$

{\rm (II)}
If in addition, we assume $p=2$ and $q\in (3,6]$;
and the initial data satisfies the compatibility condition \eqref{11113} and
\begin{gather}
\Dv(\r_0\F_0^\top)=0, \label{cc} \\
\int_{\R^3}\left(\f{1}{2}\r_0|\u_0|^2+\f{1}{2}\r_0|E_0|^2+\f{1}{\gamma-1}(\r_0^\gamma-\gamma
\r_0+\gamma-1)\right)dx\le\dl^4, \label{hh23} \\
\|\nabla\u_0\|_{L^2}+\|\u_0\cdot\nabla\u_0\|_{L^2}+\|\D\u_0\|_{L^2}+\|\nabla\r_0\|_{L^2}+\|\nabla
E_0\|_{L^2\cap L^q}\le \dl^4. \label{hh}
\end{gather}
Then, the solution   $(\r,\u, E)$ has the following improved estimates:
\begin{equation*}
\begin{cases}
\max_{t\in[0,T]}\max\big\{\|\r-1\|_{W^{1,2}\cap W^{1,q}}(t),
\|\nabla\r\|_{L^2\cap L^q}(t), \|E\|_{W^{1,2}\cap W^{1,q}}(t)\big\}\le C\dl_0,\\
\|\nabla \r\|_{L^2(0,T_0; L^q(\R^3))}\le C\dl_0,\;
\|\partial_t\u\|_{L^2(0,T_0; L^q(\R^3))}\le C\dl_0^{\f{3-\t}{2}},\;
\|\nabla E\|_{L^2(0,T_0; L^q(\R^3))}\le C\dl_0^{\f{3-\t}{2}},
\end{cases}
\end{equation*}
for some $\t\in (\f{1}{2},1]$, where the constant $C$ is independent of
the time and $\dl$.
\end{Theorem}

\begin{Remark}\label{QQ}
Under  assumption \eqref{cc}, the authors in \cite{LLZH31, LZ}
showed that the property will insist in time, that is, for all
$t\ge 0$,
$\Dv(\r\F^\top)=0.$
\end{Remark}

\begin{Remark}\label{QQ9}
It is remarkable to point out that if we could obtain one better
estimate on \\$\|\nabla \r\|_{L^2(0,T; L^q(\R^3))}$, say
$\|\nabla \r\|_{L^2(0,T; L^q(\R^3))}\le C\dl_0^{1+\alpha}$ for
some $\alpha>0$, then one could extend the local existence
in Theorem \ref{T20} to the global existence for $t\in[0,\infty)$.
\end{Remark}

\section{Local Existence} \label{local}

In this section, we  prove the local existence of strong solution in Theorem \ref{T20} using a fixed-point argument.  To this end,
 we introduce the following new variables by scaling
$$s:=\nu^2 t,\quad y:=\nu x, \quad v(y,s):=\frac{1}{\nu}\u(x,t),\quad r(y,s):=\r(x,t),\quad G(y,s):=E(x,t),$$
where $\nu>0$ will be determined later. Then, system \eqref{e1}
becomes
\begin{subequations} \label{e3}
\begin{align}
&r_t +\Dv(r v)=0,\label{e31}\\
&(rv)_t+\Dv\left(rv\otimes v\right)-\mu\D v-(\mu+\lambda)\nabla\Dv v+\nu^{-2}\nabla P=\nu^{-2}\Dv\left(r (I+G)(I+G)^\top\right),\label{e32}\\
&G_t+v\cdot\nabla G=\nabla v G+\nabla v,\label{e33}.
\end{align}
\end{subequations}
This scaling is needed in order to apply a fixed-point argument.
From \eqref{11111b},  one has
\begin{equation} \label{11111}
\nabla_{k}G_{ij}+G_{lk}\nabla_{l}G_{ij}=\nabla_{j}G_{ik}+G_{nj}\nabla_{n}G_{ik}.
\end{equation}
Thus, if we denote by $G_i$ the $i$-th row of the matrix $G$ (or the $i$-th component of the vector $G$), then \eqref{11111} becomes
\begin{equation}\label{11112}
\textrm{curl } G_i=G_{nj}\nabla_{n}G_{ik}-G_{lk}\nabla_{l}G_{ij}.
\end{equation}

The proof of local existence of strong solutions with small
initial data will be carried out  through three steps by using a
fixed point theorem. Instead of working on \eqref{e1} directly, we
will work on  \eqref{e3}.  We note that \eqref{e3} is just a
scaling version of \eqref{e1}. It can be seen from the argument
below that we only need to verify the local existence in
$\mathcal{W}^{p,q}(0,T)$, $0<T\le T_0$,  while the initial data
belongs to $V_0^{p,q}$.

\subsection{Solvability of the density with a fixed velocity}
Let $A_j(x, t)$, $j=1, ..., n$,  be symmetric $m\times m$ matrices in $\R^n\times(0,T)$,  $B(x,t)$ an $m\times m$ matrix, $f(x,t)$ and $V_0(x)$ be $m$-dimensional vector functions defined in $\R^n\times(0,T)$ and $\R^n$, respectively.

For the following initial-value problem:
\begin{equation}\label{1000}
\begin{cases}
&\displaystyle \partial_tV+\sum_{j=1}^nA_j(x,t)\partial_jV+B(x,t)V=f(x,t),\\
&V(x,0)=V_0(x),
\end{cases}
\end{equation}
we have

\begin{Lemma}\label{l1}
Assume that $$A_j\in \left[C(0,T; H^s(\R^n))\cap C^1(0,T; H^{s-1}(\R^{n}))\right]^{m\times m}, \; j=1,...,n,$$
$$B\in C((0,T), H^{s-1}(\R^n))^{m\times m},\quad f\in C((0,T), H^s(\R^n))^m,\quad V_0\in H^s(\R^n)^m,$$
with $s>\f{n}{2}+1$ is an integer. Then there exists a unique solution to \eqref{1000}, i.e, a function
$$V\in \left[C([0,T), H^s(\R^n))\cap C^1((0,T), H^{s-1}(\R^n))\right]^m$$
satisfying \eqref{1000} pointwise.
\end{Lemma}
\begin{proof}
This lemma is a direct consequence of Theorem 2.16 in \cite{AI} with $A_0(x,t)=I$.
\end{proof}

To solve the density with respect to the fixed velocity $v\in \mathcal{W}(0,T)$, we have

\begin{Lemma}\label{r}
Under the same conditions as Theorem \ref{T20}, there is a unique strictly positive function
$$r:={\mathcal S}(v)\in W^{1,p}(0,T; L^q(\R^3))\cap L^\infty(0,T; W^{1,q}(\R^3))$$
which satisfies the continuity equation \eqref{e31} and $r-1\in L^\infty(0,T; L^q(\R^3))$. Moreover, the density satisfies
\begin{equation}\label{y12}
\|\nabla r\|_{L^\infty(0,T; L^q(\R^3))}\le C(T, \|v\|_{\mathcal{W}(0,T)})\left(\|\nabla r_0\|_{L^q(\R^3)}+1\right),
\end{equation}
and the norm $\|{\mathcal S}(v)-1\|_{W^{1,q}(\R^3)}(t)$ is a continuous function in time.
\end{Lemma}

Here, and in what follows,  $C$ stands for a generic positive constant, and in some case,
we will specify its dependence on parameters by the
notation $C(\cdot)$.

\begin{proof}
For the proof of the first part of this lemma, we refer the   reader to Theorem 9.3 in \cite{AI}, or the first
part of the proof for Lemma \ref{g} below.
The positivity of density follows directly from the observations:  by writing \eqref{e31} along characteristics $\f{d}{dt}X(t)=v,$
$$\f{d}{dt}r(t,X(t))=-r(t, X(t))\Dv v(t,X(t)), \quad X(0)=x,$$
and with the help of Gronwall's inequality,
\begin{equation*}
\begin{split}
(\inf_{x}\r_0)\textrm{exp}&\left(-\int_0^t\|\Dv v(t)\|_{L^\infty(\R^3)}dx\right)\le r(t,x)
\le(\sup_{x}\r_0)
\textrm{exp}\left(\int_0^t\|\Dv v(t)\|_{L^\infty(\R^3)}dx\right).
\end{split}
\end{equation*}

Now, we can assume that the continuity equation holds pointwise  in the following form:
$$\partial_t r+r\Dv v+v\cdot \nabla r=0.$$
Taking the gradient in both sides of the above identity, multiplying by $|\nabla r|^{q-2}\nabla r$ and then integrating over $\R^3$,
we get, by Young's inequality
\begin{equation}\label{y11}
\begin{split}
\f{1}{q}\f{d}{dt}\|\nabla r\|^q_{L^q(\R^3)}&\le \int_{\R^3}|\nabla r|^q|\Dv v|dx+\int_{\R^3} r|\nabla r|^{q-1}
|\nabla\Dv v|dx\\&\qquad+\int_{\R^3}|\nabla v||\nabla r|^qdx-\f{1}{q}
\int_{\R^3} v\nabla|\nabla r|^qdx\\
&\le\|\nabla r\|^q_{L^q(\R^3)}\left(\|\nabla v\|_{L^\infty(\R^3)}+\|r\|_{L^\infty(\R^3)}\|\nabla\Dv v\|_{L^q(\R^3)}\right)\\
&\quad+\f{1}{q}\int_{\R^3}\Dv v|\nabla r|^qdx
+\|r\|_{L^\infty(\R^3)}\|\nabla\Dv v\|_{L^q(\R^3)}\\
&\le C\|\nabla r\|^q_{L^q(\R^3)}\|v\|_{W^{2,q}(\R^3)}+\|r\|_{L^\infty(\R^3)}\|\nabla\Dv v\|_{L^q(\R^3)},
\end{split}
\end{equation}
since $q>3$. Then \eqref{y12} follows from  Gronwall's inequality.

Finally, noting from \eqref{y11} and \eqref{y12} that $\f{d}{dt}\|\nabla r\|_{L^q(\R^3)}^q\in L^1(0,T)$,
and hence $$\f{d}{dt}\|\nabla (r-1)\|_{L^q(\R^3)}^q\in L^1(0,T),$$
which together  with \eqref{y12} implies that
$\|\nabla (r-1)\|^q_{L^q(\R^3)}(t)$ is continuous in time, and hence, $\|\nabla (r-1)\|_{L^q(\R^3)}(t)$ is continuous in time.
Similarly, from the continuity
equation, we know that
$$\partial_t (r-1)=-\Dv((r-1)v)-\Dv v\in L^p(0,T;L^q(\R^3)),$$
which, together with the fact $r-1\in L^\infty(0,T; L^q(\R^3))$,
yields $r-1\in C([0,T]; L^q(\R^3))$. Hence, the
quantity $\|r-1\|_{W^{1,q}(\R^3)}(t)$ is continuous in time.
The proof of Lemma \ref{r} is complete.
\end{proof}

\subsection{Solvability of the deformation gradient with a fixed velocity}\label{gradient}

Due to the hyperbolic structure of \eqref{e33}, we can apply Lemma \ref{l1} again to solve the deformation
gradient $G$ in terms of the given velocity $v\in \mathcal{W}(0,T)$. For this purpose, we have

\begin{Lemma}\label{g}
Under the same conditions as Theorem \ref{T20}, there is a unique function
$$G:={\mathcal T}(v)\in W^{1,p}(0,T; L^q(\R^3))\cap L^\infty(0,T; W^{1,q}(\R^3))$$
which satisfies the equation \eqref{e33}. Moreover, the deformation gradient satisfies
\begin{equation}\label{y31b}
\|\nabla G\|_{L^\infty(0,T; L^q(\R^3))}\le C(T,\|v\|_{\mathcal{W}(0,T)})\left(\|\nabla G(0)\|_{L^q(\R^3)}+1\right),
\end{equation}
and, the norm  $\|G\|_{W^{1,q}(\R^3)}(t)$ is a continuous function in time.
\end{Lemma}

\begin{proof}
First, we assume that $v\in C^1(0,T; C_0^\infty(\R^3)), \; G_0\in C_0^\infty(\R^3)$.
Then, we can rewrite \eqref{e33} in the component form as
$$\partial_t G_k+v\cdot\nabla G_k=\nabla v \,G_k+\nabla v_k,\quad  1\le k\le 3.$$
Applying Lemma \ref{l1}  with $A_j(x,t)=v_j(x,t) I$ for   $1\le j\le 3$, $B(x,t)=\nabla v$,
and $f(x,t)=\nabla v_k$, we get a solution
$$G\in \bigcap_{l=1}^\infty\left\{C^1(0,T, H^{l-1}(\R^3))\cap(0,T; H^l(\R^3))\right\}, $$
which implies, by the Sobolev imbedding theorem,
\begin{equation}\label{GG}
G\in \bigcap_{k=1}^\infty C^1(0,T; C^k(\R^3))=C^1(0,T; C^\infty(\R^3)).
\end{equation}

Next, for $v\in \mathcal{W}(0,T)$,  there are  two sequences:
$v_n\in C^1(0,T; C_0^\infty(\R^3)), \; G^n_0\in C_0^\infty(\R^3),$
such that $ v_n\to   v  \text{ in } \mathcal{W}(0,T), \;  G^n_0\to   G_0 \text{ in }W^{1,q}(\R^3),$
thus $v_n\to   v$ in $C(B(0,a)\times(0,T))$ for all $a>0$ where $B(0,a)$ denotes the ball with radius $a$ and
centered at the origin. According to the previous result \eqref{GG}, there are a sequence of functions
$\{G_n\}_{n=1}^\infty\subset C^1(0,T; C^\infty(\R^3))$ satisfying
\begin{equation}\label{2}
\partial_t G_n+v_n\cdot\nabla G_n=\nabla v_n\,G_n+\nabla v_n,
\end{equation}
with $G_n(0)=G^n_0$.  Multiplying
\eqref{2} by $|G_n|^{q-2}G_n$, and integrating over $\R^3$, using
integration by parts and Young's inequality, we obtain,
\begin{equation*}
\begin{split}
&\f{1}{q}\f{d}{dt}\int_{\R^3}|G_n|^q dx\\
&=-\f{1}{q}\int_{\R^3} v_n\cdot\nabla |G_n|^q dx
+\int_{\R^3}\nabla v_n|G_n|^{q-2}G^2_n dx+\int_{\R^3}\nabla v_n |G_n|^{q-2} G_ndx\\
&\le \f{1+q}{q}\|G_n\|_{L^q}^q(\|\nabla v_n\|_{L^\infty}+\|\nabla v_n\|_{L^q})+\|\nabla v_n\|_{L^q}.
\end{split}
\end{equation*}
From  Gronwall's inequality, one obtains,
\begin{equation*}
\begin{split}
&\int_{\R^3}|G_n|^q dx\\
&\le\left\{\int_{\R^3}|G_n(0)|^q dx+q\int_0^t\|\nabla v_n\|_{L^q}
\textrm{exp}\left(-\int_0^t(q+1)(\|\nabla v_n\|_{L^\infty}+\|\nabla v_n\|_{L^q})d\tau\right)ds\right\}\\
&\qquad\times\textrm{exp}\left(\int_0^t(q+1)(\|\nabla v_n\|_{L^\infty}+\|\nabla v_n\|_{L^q})ds\right)\\
&\le\left(\int_{\R^3}|G_n(0)|^q dx+q\int_0^t\|\nabla v_n\|_{L^q}ds\right)\textrm{exp}
\left(\int_0^t(q+1)(\|\nabla v_n\|_{L^\infty}+\|\nabla v_n\|_{L^q})ds\right).
\end{split}
\end{equation*}
Thus,
\begin{equation}\label{4}
\|G_n\|_{L^\infty(0,T; L^q(\R^3))}\le C(T, \|v\|_{L^p(0,T; W^{2,q}(\R^3))})
\left(\|G(0)\|_{L^q(\R^3)} +1\right)<\infty.
\end{equation}
Hence, up to a subsequence, we can assume that the sequence $\{v_n\}$ was chosen
so that
$G_n\to   G\quad \textrm{weak-* in}\quad L^\infty(0,T; L^q(\R^3)).$

Taking the gradient in both sides of \eqref{2}, multiplying by $|\nabla G_n|^{q-2}\nabla G_n$ and
then integrating over $\R^3$, we get, with the help of H\"{o}lder's inequality and Young's inequality,
\begin{equation}\label{y3}
\begin{split}
&\f{1}{q}\f{d}{dt}\|\nabla G_n\|^q_{L^q(\R^3)}\\
&\le \int_{\R^3}|\nabla G_n|^q|\nabla v_n|dx
+\int_{\R^3} |G_n||\nabla G_n|^{q-1}|\nabla\nabla v_n|dx\\&\qquad+\int_{\R^3}|\nabla v_n||\nabla G_n|^qdx-\f{1}{q}
\int_{\R^3} v_n\nabla|\nabla G_n|^qdx +\int_{\R^3}|\nabla\nabla v_n||\nabla G_n|^{q-1}dx\\
&\le \int_{\R^3}|\nabla G_n|^q|\nabla v_n|dx+\int_{\R^3} |G_n||\nabla G_n|^{q-1}|\nabla\nabla v_n|dx\\
&\qquad+\int_{\R^3}|\nabla v_n||\nabla G_n|^qdx+\f{1}{q}
\int_{\R^3} |\nabla v_n||\nabla G_n|^qdx +\int_{\R^3}|\nabla\nabla v_n||\nabla G_n|^{q-1}dx\\
&\le C\|\nabla G_n\|^q_{L^q(\R^3)}\|v_n\|_{W^{2,q}(\R^3)}+(\|G_n\|^q_{L^\infty(\R^3)}+1)\|v_n\|_{W^{2,q}(\R^3)}\\
&\le C\|\nabla G_n\|^q_{L^q(\R^3)}\|v_n\|_{W^{2,q}(\R^3)}+(\|G_n\|^q_{L^q(\R^3)}+1)\|v_n\|_{W^{2,q}(\R^3)},
\end{split}
\end{equation}
since $q>3$. Using Gronwall's inequality and \eqref{4}, we
conclude from \eqref{y3} that
$$\|\nabla G_n\|_{L^\infty(0,T; L^q(\R^3))}\le C(T,\|v_n\|_{\mathcal{W}(0,T)})\left(\|\nabla G_n(0)\|_{L^q(\R^3)}+1\right),$$
and hence,
\begin{equation}\label{y31}
\begin{split}
\|\nabla G\|_{L^\infty(0,T; L^q(\R^3))}&\le\liminf_{n\to  \infty}\|\nabla G_n\|_{L^\infty(0,T; L^q(\R^3))}\\
&\le C(T,\|v\|_{\mathcal{W}(0,T)})\left(\|\nabla G(0)\|_{L^q(\R^3)}+1\right).
\end{split}
\end{equation}
Therefore,
$$\|G_n\|_{L^\infty(0,T; W^{1,q}(\R^3))}\le C(T, \|v\|_{L^p(0,T; W^{2,q}(\R^3))},\|G(0)\|_{W^{1,q}(\R^3)})<\infty.$$
Furthermore, since $q>3$, we deduce $G\in L^\infty(Q_T)$ and
\begin{equation*}
\|G\|_{L^\infty(Q_T)}\le C(T, \|v\|_{L^p(0,T; W^{2,q}(\R^3))}, \|G(0)\|_{W^{1,q}(\R^3)}
)<\infty.
\end{equation*}
Passing to the limit as $n\to  \infty$ in \eqref{2}, we show that \eqref{e33} holds at
least in the sense of distributions. Therefore,
$\partial_t G\in L^p(0,T; L^q(\R^3)),$
then $G\in W^{1,p}(0,T; L^q(\R^3))$, and hence $G\in C([0,T]; L^q(\R^3))$.

Finally, to show that the quantity $\|G\|_{W^{1,q}(\R^3)}(t)$ is
continuous in time, it suffices to show that $\|\nabla
G\|_{L^q(\R^3)}$ is \ continuous in time. Indeed, from \eqref{y3},
we know that $$\f{d}{dt}\|\nabla G\|^q_{L^q(\R^3)}(t)\in L^p(0,T),$$
which, with \eqref{y31}, implies that $\|\nabla G\|_{L^q(\R^3)}\in C([0,T]).$
The proof of Lemma \ref{g} is complete.
\end{proof}

\subsection{Local existence via the fixed-point theorem}

In order to solve locally system \eqref{e3}, we need to use the following fixed point theorem (cf. 1.4.11.6 in \cite{AI}):

\begin{Theorem}[Tikhonov Theorem]  \label{T2}
Let $M$ be a nonempty bounded closed convex subset of a separable reflexive Banach space $X$ and
let $F: M\mapsto M$ be a weakly continuous mapping (i.e. if $x_n\in M, x_n\to   x$ weakly in $X$,
then $F(x_n)\to   F(x)$ weakly in $X$ as well). Then $F$ has at least one fixed point in $M$.
\end{Theorem}

Now, let us consider the following operator
$$L\omega:=\omega_t-\mu\D\omega-(\mu+\lambda)\nabla\Dv\omega,\quad \omega\in \mathcal{W}^{p,q}(0,T).$$
One has the following theorem by the maximal regularity of parabolic equations; see
Theorem 9.2 in \cite{AI}, or equivalently Theorem 4.10.7 and
Remark 4.10.9 in \cite{HA} (page 188).

\begin{Theorem}\label{T3}
Given $1<p<\infty$, $\omega_0\in V_0^{p,q}$ and $f\in L^p(0,T; L^q(\R^3)^3)$, the Cauchy problem
$$L\omega=f,\quad t\in(0,T); \quad \omega(0)=\omega_0,$$
has a unique solution $\omega:=L^{-1}(\omega_0, f)\in \mathcal{W}^{p,q}(0,T)$, and
$$\|\omega\|_{\mathcal{W}^{p,q}(0,T)}\le C\left(\|f\|_{L^p(0,T; L^q(\R^3))}+\|\omega_0\|_{V_0^{p,q}}\right),$$
where $C$ is independent of $\omega_0$, $f$ and $T$. Moreover, there exists a positive constant $c_0$ independent of $f$ and $T$ such that
$$\|\omega\|_{\mathcal{W}^{p,q}(0,T)}\ge c_0\sup_{t\in(0,T)}\|\omega(t)\|_{V^{p,q}_0}.$$
\end{Theorem}

Notice that Theorem \ref{T3} implies that the operator $L$ is
invertible. Thus we define the operator ${\mathcal H}(v):
\mathcal{W}^{p,q}(0,T)\mapsto\mathcal{W}^{p,q}(0,T)$ by
\begin{multline}
{\mathcal H}(v):=L^{-1}\Big(v_0, \; \partial_t((1-{\mathcal S}(v))v)-\Dv({\mathcal S}(v)v\otimes
v)+\nu^{-2}\nabla(P(1)-P({\mathcal S}(v)))\\
+ \nu^{-2}\Dv({\mathcal S}(v)(I+{\mathcal T}(v))(I+{\mathcal T}(v))^\top)\Big).
\end{multline}
Then, solving  system \eqref{e3} is equivalent to solving
\begin{equation}\label{5}
v={\mathcal H}(v).
\end{equation}

To solve \eqref{5}, we define
$$B_R(0):=\{v\in \mathcal{W}^{p,q}(0,T):\, \|v\|_{\mathcal{W}^{p,q}(0,T)}\le R\}.$$
Then, we prove first the following claim:

\begin{Lemma}\label{H1}
There are $\nu, T>0$, and $0<R<1$ such that $${\mathcal H}(B_R(0))\subset B_R(0).$$
\end{Lemma}

\begin{proof}
Let $T>0$, $0<R<1$ and $v\in B_R(0)$. Since ${\mathcal S}(v)$ solves \eqref{e31}, we can rewrite  operator ${\mathcal H}$ as
\begin{multline}
{\mathcal H}(v)=L^{-1}\Big(v_0, \; (1-{\mathcal S}(v))\partial_t v-{\mathcal S}(v)v\cdot\nabla
v+\nu^{-2}\nabla(P(1)-P({\mathcal S}(v)))\\+
\nu^{-2}\Dv({\mathcal S}(v)(I+{\mathcal T}(v))(I+{\mathcal T}(v))^\top)\Big).
\end{multline}
Thus, it suffices to prove that the terms in the above expression are small in the norm of $L^p(0,T; (L^q(\R^3))^3)$.

First of all, we begin to deal with the first term by letting $\ov{r}:={\mathcal S}(v)-1$. Thus, $\ov{r}$ satisfies the equations
\begin{equation*}
\partial_t\ov{r}+\Dv(\ov{r}v)=0,\quad
\ov{r}(x,0)=r_0-1.
\end{equation*}
Repeating the argument in Section \ref{gradient} again, we obtain
\begin{equation*}
\begin{split}
\|\ov{r}\|_{L^\infty(Q_T)}&\le C\|\ov{r}\|_{L^\infty(0,T;W^{1,q}(\R^3))}
\le C\|r_0-1\|_{W^{1,q}(\R^3)}C(T, \|v\|_{\mathcal{W}(0,T)})\\
&\le \|r_0-1\|_{W^{1,q}(\R^3)} C(T,R)\le C(T)R,
\end{split}
\end{equation*}
where, by the formula of change of variables, we deduce that
$$\|r_0-1\|_{L^q(\R^3)}\le \nu^{\f{3}{q}}\|\r_0-1\|_{L^q(\R^3)}\le R,\quad
\|\nabla r_0\|_{L^q(\R^3)}\le \nu^{\f{3-q}{q}}\|\nabla \r_0\|_{L^q(\R^3)}\le R,$$
if $\|\r_0-1\|_{L^q(\R^3)}$ is small
enough and $\nu>1$ is large enough. Hence, due to the assumption
$v\in B_R(0)$, we obtain
\begin{equation}\label{7}
\|(1-{\mathcal S}(v))\partial_t v\|_{L^p(0,T; L^q(\R^3))}\le C(T)R^2.
\end{equation}

Secondly, by the Sobolev imbedding,
$$\int_{\R^3}|v\nabla v|^qdx\le \|v\|^q_{L^q(\R^3)}\|\nabla v\|^q_{L^\infty(\R^3)}\le C\|v\|_{L^q(\R^3)}^q\|v\|_{W^{2,q}(\R^3)}^q,$$
and thus, since $W^{1,p}(0,T;L^q(\R^3))\hookrightarrow C([0,T]; L^q(\R^3))$, we deduce
\begin{equation*}
\begin{split}
\int_0^T\left(\int_{\R^3}|v\nabla v|^qdx\right)^{\f{p}{q}}ds&\le C\int_0^T\|v\|_{L^q(\R^3)}^p\|v\|_{W^{2,q}(\R^3)}^pds\\
&\le C\|v\|_{L^\infty(0,T;L^q(\R^3))}^p\|v\|_{L^p(0,T; W^{2,q}(\R^3))}^p \le \|v\|^{2p}_{\mathcal{W}(0,T)}\le CR^{2p}.
\end{split}
\end{equation*}
Therefore, we get
\begin{equation}\label{8}
\|{\mathcal S}(v)(v\cdot\nabla)v\|_{L^p(0,T; L^q(\R^3))}\le CR^2.
\end{equation}

Thirdly, for the term $\nabla P({\mathcal S}(v))$, we can estimate it as follows
\begin{equation}\label{9}
\begin{split}
&\|\nabla P({\mathcal S}(v))\|_{L^p(0,T; L^q(\R^3))} \\
&\le C(T)\sup\left\{P'(\eta): C(T)^{-1}\le\eta\le C(T)\right\}(\|\nabla r_0\|_{L^q(\R^3)}+1).
\end{split}\end{equation}

Fourthly, for the term $\Dv ({\mathcal S}(v)(I+{\mathcal T}(v))(I+{\mathcal T}(v))^\top$, we have
$$|\Dv ({\mathcal S}(v)(I+{\mathcal T}(v))(I+{\mathcal T}(v))^\top)|\le |\nabla {\mathcal S}(v)||I+{\mathcal T}(v)|^2+2{\mathcal S}(v)|\nabla {\mathcal T}(v)||I+{\mathcal T}(v)|,$$
and hence,
\begin{equation}\label{10}
\begin{split}
&\|\Dv ({\mathcal S}(v)(I+{\mathcal T}(v))(I+{\mathcal T}(v))^\top)\|_{L^p(0,T; L^q)}\\
&\le \||\nabla {\mathcal S}(v)||I+{\mathcal T}(v)|^2\|_{L^p(0,T; L^q)}+2\|{\mathcal S}(v)|\nabla {\mathcal T}(v)||I+{\mathcal T}(v)|\|_{L^p(0,T; L^q)}\\
&\le C(T)M,
\end{split}
\end{equation}
with
$$M=\max\left\{\|G_0\|_{W^{1,q}}+1, \|G_0\|_{L^\infty(\R^3)}+1, \|r_0\|_{W^{1,q}}+1, \|r_0\|_{L^\infty(\R^3)}+1\right\}^3<\infty.$$

Combining together \eqref{7}-\eqref{10}, using the Theorem \ref{T3},
and assuming  parameter $\nu$ sufficiently large and $R<1$ sufficiently small, we get
$$\|{\mathcal H}(v)\|_{\mathcal{W}(0,T)}\le C(T)(R^2+\nu^{-2})\le R.$$
The proof of Lemma \ref{H1} is complete.
\end{proof}

Thus, it is only left to show the following:

\begin{Lemma}\label{H2}
The operator ${\mathcal H}$ is weakly continuous from
$\mathcal{W}^{p,q}(0,T)$ into itself.
\end{Lemma}
\begin{proof}
Assume that $v_n\to   v$ weakly in $\mathcal{W}^{p,q}(0,T)$, and set
$r_n:={\mathcal S}(v_n),\, G_n:={\mathcal T}(v_n),$
 then $\{r_n\}_{n=1}^\infty$ and $\{G_n\}_{n=1}^\infty$ are uniformly bounded in
 $L^\infty(0,T; W^{1,q}(\R^3))\cap W^{1,p}(0, T; L^q(\R^3))$ by Lemmas \ref{l1} and \ref{g}.
Hence, up to a subsequence, we can assume that
$r_n\to   r$ and $G_n\to   G$ $\textrm{weakly}^*$ in $L^\infty(0,T; W^{1,q}(\R^3))\cap W^{1,p}(0, T; L^q(\R^3))$
and then strongly in $C((0,T)\times B(0,a))$ for all $a>0$.
And at least the same convergence holds for $v_n$. Thus, \eqref{e31} and \eqref{e33}
follow easily  from above convergence.

Since $r_n\to   r$ weakly* in $L^\infty(0,T; W^{1,q}(\R^3))\cap W^{1,p}(0, T; L^q(\R^3))$,
we can assume that $P'({\mathcal S}(v_n))\nabla {\mathcal S}(v_n)\to   P'({\mathcal S}(v))\nabla {\mathcal S}(v)$ weakly in $L^p(0,T; L^q(\R^3))$ and hence,
$$L^{-1}\left(0, \nabla P({\mathcal S}(v_n))\right)\to   L^{-1}
\left(0, \nabla P({\mathcal S}(v))\right)\quad \textrm{weakly in }\quad\mathcal{W}(0,T),$$
since the strong continuity of $L^{-1}$ from $L^p(0,T; L^q(\R^3))$ into $\mathcal{W}(0,T)$
and the linearity of the operator $L$ imply also the weak continuity in these spaces.

Similarly, since $\partial_t v_n\to   \partial_t v$ weakly in $L^p(0,T; L^q(\R^3))$
and $r_n\to   r$ in $C((0,T)\times B(0,a))$ for all $a>0$, we have $(r_e-r_n)\partial_t v_n\to   (r_e-r)\partial_t v$
weakly in $L^p(0,T; L^q(\R^3))$ and consequently
$$L^{-1}\left(0, (r_e-r_n)\partial_t v_n\right)\to   L^{-1}\left(0, (r_e-r)
\partial_t v\right)\quad \textrm{weakly in }\quad \mathcal{W}(0,T).$$

Since $\nabla v_n\to   \nabla v$ weakly in $W^{1,p}(0,T; W^{-1,q}(\R^3))\cap L^p(0,T; W^{1,q}(\R^3))$
which is compactly imbedded in to $C([0,T]; L^q(B(0,a)))$ for all $a>0$, we can assume that $v_n\to   v$
strongly in $L^\infty(0,T; L^q(B(0,a)))$ for all $a>0$, and then
$${\mathcal S}(v_n)(v_n\cdot\nabla) v_n\to   {\mathcal S}(v)(v\cdot\nabla)v$$weakly in $L^p(0,T; L^q(\R^3)).$
Hence
$$L^{-1}\left(0, {\mathcal S}(v_n)(v_n\cdot\nabla)v_n\right)\to
L^{-1}\left(0, {\mathcal S}(v)(v\cdot\nabla) v\right)\quad \textrm{weakly in }\quad \mathcal{W}(0,T).$$

Finally, due to the facts that $r_n\to   r$ and $G_n\to   G$ $\textrm{weakly}^*$
in $L^\infty(0,T; W^{1,q}(\R^3))\cap W^{1,p}(0, T; L^q(\R^3))$ and strongly in $C((0,T)\times B(0,a))$ for all $a>0$, we deduce that
$$\Dv({\mathcal S}(v_n)(I+{\mathcal T}(v_n))(I+{\mathcal T}(v_n))^\top)\to  \Dv({\mathcal S}(v)(I+{\mathcal T}(v))
(I+{\mathcal T}(v))^\top)$$ weakly in $L^p(0,T; L^q(\R^3)).$
Therefore,
$$L^{-1}\left(0, \Dv({\mathcal S}(v_n)(I+{\mathcal T}(v_n))(I+{\mathcal T}(v_n))^\top)\right)\to   L^{-1}
\left(0, \Dv({\mathcal S}(v)(I+{\mathcal T}(v))(I+{\mathcal T}(v))^\top)\right)$$ weakly in  $\mathcal{W}(0,T)$.

Thus, we can conclude that
$${\mathcal H}(v_n)\to   {\mathcal H}(v)\quad \textrm{weakly in}\quad\mathcal{W}(0,T),$$
due to  the weak continuity of  map $L^{-1}(v, 0)$. The proof of Lemma \ref{H2} is complete.
\end{proof}

Therefore, by Theorem \ref{T2}, there exists at least a fixed point
\begin{equation}\label{11}
v={\mathcal H}(v)\in B_R(0)\subset\mathcal{W}(0,T),
\end{equation}
of  mapping ${\mathcal H}$. The fixed point $v$ provides a local in time solution $(\r, \u, E)$ of
system \eqref{e1} near its equilibrium through the scaling with $\nu$ sufficiently large.

The proof of the local existence in Theorem \ref{T20} is complete. The uniqueness will be proved in the next section.

\section{Uniqueness}

In this section, we  prove the uniqueness of the local solution found in the previous section.
Notice that, the argument in Section \ref{local}  yields  that
$$\partial_t v\in L^2(0,T; L^2(\R^3),\quad \nabla r\in L^2(0,T; L^2(\R^3)),\quad \nabla G\in L^2(0,T; L^2(\R^3)).$$
Hence, using the interpolation, we deduce that
$$\partial_t v\in L^{p_0}(0,T; L^3(\R^3),\quad \nabla r\in L^{p_0}(0,T; L^3(\R^3)),\quad
\nabla G\in L^{p_0}(0,T; L^3(\R^3)),$$ where
$$\f{1}{p_0}=\f{\theta}{2}+\f{1-\theta}{p},\quad
\f{1}{3}=\f{\theta}{2}+\f{1-\theta}{q},$$ for some $\theta\in
[0,1]$. Now, assume that $v_1$, $v_2$ satisfy \eqref{11} for some
$T>0$. Let
$$r:={\mathcal S}(v_1)-{\mathcal S}(v_2),\quad v:=v_1-v_2,\quad G:={\mathcal T}(v_1)-{\mathcal T}(v_2),$$
 with a little abuse of notations (however, there should be no confusion in the rest of this section). Then, we have
\begin{equation}\label{12}
\partial_t r+v_1\cdot\nabla r+v\cdot\nabla {\mathcal S}(v_2)+r\Dv v_1+{\mathcal S}(v_2)\Dv v=0,\quad
r(0)=0.
\end{equation}
Multiplying \eqref{12} by $r$, and integrating over $\R^3$, we get
\begin{equation*}
\begin{split}
\f{1}{2}\f{d}{dt}\|r\|_{L^2}^2-\f{1}{2}\int_{\R^3}|r|^2\Dv v_1dx
+\int_{\R^3} \left(v\nabla {\mathcal S}(v_2)r+|r|^2\Dv v_1+ r{\mathcal S}(v_2)\Dv v\right)dx=0,
\end{split}
\end{equation*}
which yields
\begin{equation}\label{1213}
\begin{split}
\f{d}{dt}\|r\|^2_{L^2(\R^3)}&\le \|\Dv v_1\|_{L^\infty}\|r\|^2_{L^2}+\varepsilon\|\nabla v\|_{L^2}^2
+C(\varepsilon)\|\nabla {\mathcal S}(v_2)r\|^2_{L^{\f{6}{5}}}\\&\qquad+\varepsilon\|\nabla v\|^2_{L^2(\R^3)}
+C(\varepsilon)\|{\mathcal S}(v_2)\|_{L^\infty}^2\|r\|^2_2\\
&\le \|\Dv v_1\|_{L^\infty}\|r\|^2_{L^2}+\varepsilon\|\nabla v\|_{L^2}^2+C(\varepsilon)
\|\nabla {\mathcal S}(v_2)\|^2_{L^3}\|r\|_{L^2}^2\\&\qquad+\varepsilon\|\nabla v\|^2_{L^2(\R^3)}
+C(\varepsilon)\|{\mathcal S}(v_2)\|_{L^\infty}^2\|r\|^2_2\\
&\le\eta_1(\varepsilon)\|r\|_{L^2}^2+2\varepsilon\|\nabla v\|^2_{L^2(\R^3)},
\end{split}
\end{equation}
where $\varepsilon>0$, $\eta_1(\varepsilon)=\|\Dv v_1\|_{L^\infty}
+C(\varepsilon)\left(\|\nabla {\mathcal S}(v_2)\|^2_{L^3}+\|{\mathcal S}(v_2)\|^2_{L^\infty}\right).$

Similarly, from \eqref{e33}, we obtain
\begin{equation}\label{1211}
\partial_t G+v_1\cdot\nabla G+v\cdot\nabla G_2=\nabla v_1 G+\nabla v G_2+\nabla v,\quad
G(0)=0.
\end{equation}
Multiplying \eqref{1211} by $G$, and integrating over $\R^3$, we get
\begin{equation*}
\begin{split}
&\f{1}{2}\f{d}{dt}\|G\|_{L^2}^2-\f{1}{2}\int_{\R^3}|G|^2\Dv v_1dx+\int_{\R^3} v\cdot\nabla {\mathcal T}(v_2):Gdx \\
&=\int_{\R^3}|G|^2\nabla v_1dx+\int_{\R^3} \nabla v{\mathcal T}(v_2): Gdx
+\int_{\R^3} \nabla v :Gdx,
\end{split}
\end{equation*}
which yields
\begin{equation}\label{1212}
\begin{split}
\f{d}{dt}\|G\|^2_{L^2(\R^3)}&\le \|\Dv v_1\|_{L^\infty}\|G\|^2_{L^2}+\varepsilon\|\nabla v\|_{L^2}^2
+C(\varepsilon)\|\nabla {\mathcal T}(v_2)G\|^2_{L^{\f{6}{5}}}\\&\qquad+\varepsilon\|\nabla v\|^2_{L^2(\R^3)}
+C(\varepsilon)\|{\mathcal T}(v_2)\|_{L^\infty}^2\|G\|^2_2+\varepsilon\|\nabla v\|^2_{L^2}+C(\varepsilon)
\|G\|^2_{L^2}\\
&\le \|\Dv v_1\|_{L^\infty}\|G\|^2_{L^2}+\varepsilon\|\nabla v\|_{L^2}^2+C(\varepsilon)
\|\nabla {\mathcal T}(v_2)\|^2_{L^3}\|G\|_{L^2}^2\\&\qquad+\varepsilon\|\nabla v\|^2_{L^2(\R^3)}
+C(\varepsilon)\|{\mathcal T}(v_2)\|_{L^\infty}^2\|G\|^2_2+\varepsilon\|\nabla v\|^2_{L^2}+C(\varepsilon)\|G\|^2_{L^2}\\
&\le\eta_2(\varepsilon)\|G\|_{L^2}^2+3\varepsilon\|\nabla v\|^2_{L^2(\R^3)},
\end{split}
\end{equation}
where $\eta_2(\varepsilon)=\|\Dv v_1\|_{L^\infty}+C(\varepsilon)\left(\|\nabla {\mathcal T}(v_2)\|^2_{L^3}
+\|{\mathcal T}(v_2)\|^2_{L^\infty}+1\right).$

For each $v_j$, $j=1, 2$, we deduce from \eqref{e32} that
\begin{equation*}
\begin{cases}
&{\mathcal S}(v_j)\partial_t v_j-\mu\Delta v_j-(\mu+\lambda)\nabla\Dv v_j\\
&\qquad\quad =-{\mathcal S}(v_j)
(v_j\cdot\nabla)v_j-\nabla P({\mathcal S}(v_j))+\Dv({\mathcal S}(v_j)(I+{\mathcal T}(v_j))(I+{\mathcal T}(v_j))^\top),\\
&v_j(0)=v_0.
\end{cases}
\end{equation*}
Subtracting these equations, we obtain,
\begin{equation}\label{13}
\begin{split}
&{\mathcal S}(v_1)\partial_t v_1-{\mathcal S}(v_2)\partial_t v_2-\mu\D v-(\mu+\lambda)\nabla\Dv v\\
&=-{\mathcal S}(v_1)(v_1\cdot\nabla)v_1+{\mathcal S}(v_2)(v_2\cdot\nabla)v_2
-\nabla P({\mathcal S}(v_1))+\nabla P({\mathcal S}(v_2))\\
&\quad+\Dv({\mathcal S}(v_1)(I+{\mathcal T}(v_1))(I+{\mathcal T}(v_1))^\top)-\Dv({\mathcal S}(v_2)(I+{\mathcal T}(v_2))(I+{\mathcal T}(v_2))^\top).
\end{split}
\end{equation}
Since
\begin{equation*}
\begin{split}
&-{\mathcal S}(v_1)(v_1\cdot\nabla)v_1+{\mathcal S}(v_2)(v_2\cdot\nabla)v_2\\&=-{\mathcal S}(v_1)(v\cdot\nabla)v_1-
({\mathcal S}(v_1)-{\mathcal S}(v_2))(v_2\cdot\nabla)v_1-{\mathcal S}(v_2)(v_2\cdot\nabla)v,
\end{split}
\end{equation*}
and
\begin{equation*}
\begin{split}
&{\mathcal S}(v_1)(I+{\mathcal T}(v_1))(I+{\mathcal T}(v_1))^\top-{\mathcal S}(v_2)(I+{\mathcal T}(v_2))(I+{\mathcal T}(v_2))^\top\\&={\mathcal S}(v_1)G(I+{\mathcal T}(v_1))^\top
+r(I+{\mathcal T}(v_2))(I+{\mathcal T}(v_1))^\top+{\mathcal S}(v_2)(I+{\mathcal T}(v_2))G^\top,
\end{split}
\end{equation*}
we can rewrite \eqref{13} as
\begin{equation}\label{14}
\begin{split}
&{\mathcal S}(v_1)\partial_t v-\mu\D v-(\mu+\lambda)\nabla\Dv v
\\&=-r\partial_t v_2-{\mathcal S}(v_1)(v\cdot\nabla)
v_1-{\mathcal S}(v_2\cdot\nabla)v_1-{\mathcal S}(v_2)(v_2\cdot\nabla)v-\nabla P({\mathcal S}(v_1))+\nabla P({\mathcal S}(v_2))\\
&\quad+\Dv({\mathcal S}(v_1)G(I+{\mathcal T}(v_1))^\top+r(I+{\mathcal T}(v_2))(I+{\mathcal T}(v_1))^\top+{\mathcal S}(v_2)(I+{\mathcal T}(v_2))G^\top).
\end{split}
\end{equation}
Multiplying \eqref{14} by $v$, using the continuity equation \eqref{e11} and integrating over $\R^3$, we deduce that
\begin{equation}\label{15}
\begin{split}
&\f{1}{2}\f{d}{dt}\int_{\R^3} {\mathcal S}(v_1)|v|^2 dx+\int_{\R^3}\left(\mu|\nabla v|^2+(\mu+\lambda)|\Dv v|^2\right)dx\\
&=\int_{\R^3} \f{1}{2}{\mathcal S}(v_1)(v_1\cdot\nabla)v\cdot v-r\partial_t v_2v-{\mathcal S}(v_1)(v\cdot\nabla)
v_1v-{\mathcal S}(v_2\cdot\nabla)v_1v\\&\quad-{\mathcal S}(v_2)(v_2\cdot\nabla)vv-\nabla P({\mathcal S}(v_1))v+\nabla
P({\mathcal S}(v_2))v\\&\quad-({\mathcal S}(v_1)G(I+{\mathcal T}(v_1))^\top+r(I+{\mathcal T}(v_2))(I+{\mathcal T}(v_1))^\top+{\mathcal S}(v_2)(I+{\mathcal T}(v_2))G^\top)\nabla vdx\\
&\le \varepsilon \|\nabla v\|^2_{L^2}+C(\varepsilon)\|{\mathcal S}(v_1)\|^2_{L^\infty}
\|v_1\|^2_{L^\infty}\|v\|^2_{L^2}+\varepsilon\|\nabla v\|^2_{L^2}+C(\varepsilon)\|\partial_t v_2\|_{L^3}^2\|r\|_{L^2}^2\\
&\quad +\|{\mathcal S}(v_1)\|_{L^\infty}\|\nabla v_1\|_{L^\infty}\|v\|^2_{L^2}
+2\|v_2\|_{L^\infty}\|\nabla v_1\|_{L^\infty}(\|r\|_{L^2}^2+\|v\|^2_{L^2})\\
&\quad +\varepsilon\|\nabla v\|^2_{L^2}+C(\varepsilon)\|{\mathcal S}(v_2)\|_{L^\infty}^2
\|v_2\|_{L^\infty}^2\|v\|_{L^2}^2+\varepsilon\|\nabla v\|_{L^2}^2\\
&\quad+C(\varepsilon)(\sup\{P'(\eta): C(T)^{-1}\le\eta\le C(T)\})^2\|r\|^2_{L^2}
+\varepsilon\|\nabla v\|^2_{L^2}\\
&\quad+C(\varepsilon)(\|{\mathcal S}(v_1)\|_{L^\infty}^2(1+\|{\mathcal T}(v_1)\|^2_{L^\infty})\|G\|^2_{L^2}\\
&\quad+\|{\mathcal S}(v_2)\|_{L^\infty}^2(1+\|{\mathcal T}(v_2)\|^2_{L^\infty})\|G\|^2_{L^2}
+\|r\|^2_{L^2}(1+\|{\mathcal T}(v_1)\|_{L^\infty}^2)(1+\|{\mathcal T}(v_2)\|_{L^\infty}^2))\\
&\le 5\varepsilon\|\nabla v\|^2_{L^2}+\eta_3(\varepsilon)(\|r\|^2_{L^2}+\|v\|^2_{L^2}+\|G\|^2_{L^2})
\end{split}
\end{equation}
with
\begin{equation*}
\begin{split}
\eta_3(\varepsilon)&=C(\varepsilon)\|{\mathcal S}(v_1)\|^2_{L^\infty}\|v_1\|^2_{L^\infty}+C(\varepsilon)\|\partial_t v_2\|_{L^3}^2+
\|{\mathcal S}(v_1)\|_{L^\infty}\|\nabla v_1\|_{L^\infty}\\&\quad+2\|v_2\|_{L^\infty}\|\nabla v_1\|_{L^\infty}
+C(\varepsilon)\|{\mathcal S}(v_2)\|_{L^\infty}^2\|v_2\|_{L^\infty}^2\\
&\quad+C(\varepsilon)(\sup\{P'(\eta): C(T)^{-1}\le\eta\le C(T)\})^2\\&\quad+C(\varepsilon)(\|{\mathcal S}(v_1)
\|_{L^\infty}^2(1+\|{\mathcal T}(v_1)\|^2_{L^\infty})+\|{\mathcal S}(v_2)\|_{L^\infty}^2(1+\|{\mathcal T}(v_2)\|^2_{L^\infty})
\\&\quad+(1+\|{\mathcal T}(v_1)\|_{L^\infty}^2)(1+\|{\mathcal T}(v_2)\|_{L^\infty}^2)).
\end{split}
\end{equation*}

Summing up \eqref{1213}, \eqref{1212}, and \eqref{15}, by taking $\varepsilon=\f{\mu}{20}$, we obtain
\begin{equation}\label{16}
\begin{split}
&\f{d}{dt}\int_{\R^3} ({\mathcal S}(v_1)|v|^2+|r|^2+|G|^2)dx+\int_{\R^3}\left(\mu|\nabla v|^2+(\mu+\lambda)|\Dv v|^2\right)dx\\
&\le 2(\eta_3(\varepsilon)+\eta_2(\varepsilon)+\eta_1(\varepsilon))(\|v\|_{L^2}^2+\|r\|^2_{L^2}+\|G\|^2_{L^2})\\
&\le 2\eta(\varepsilon,t)\int_{\R^3}({\mathcal S}(v_1)|v|^2+|r|^2+|G|^2)dx,
\end{split}
\end{equation}
with
$$\eta(\varepsilon, t)=\f{\eta_3(\varepsilon)+\eta_2(\varepsilon)+\eta_1(\varepsilon)}{\min\{\min_{x\in\R^3}{\mathcal S}(v_1)(x,t),1\}}.$$
It is a routine matter to establish the integrability with respect to $t$ of the function $\eta(\varepsilon,t)$
on the interval $(0,T)$. This is a consequence of the regularity of $v_1, v_2\in \mathcal{W}(0,T)$
and the estimates in Lemmas \ref{r} and \ref{g} for ${\mathcal S}(v_i)$, ${\mathcal T}(v_i)$ with $i=1,2.$
Therefore, \eqref{16}, combining with Gronwall's inequality, implies
\begin{equation}\label{uuuuu}
\int_{\R^3}\left({\mathcal S}(v_1)|v|^2+|r|^2+|G|^2\right)dx=0,\quad \textrm{for all}\quad t\in(0,T),
\end{equation}
and consequently
$v\equiv 0,\; r\equiv 0,\; G\equiv 0.$
Thus, the uniqueness in Theorem \ref{T20} is established.

\section{Uniform \textit{A Priori} Estimates} \label{estimates}

Up to now, we prove that for any given $T_0$, we can find a unique
solution to the scaling system \eqref{e3}.  That is, we have
proved the local existence of solution to the viscoelastic fluid
system \eqref{e1} and its uniqueness. For the unique solution we
constructed in the previous sections, we have some
$\textit{a priori}$ estimates uniform in time as stated in Theorem \ref{T1}.
In this section, we prove the first part of Theorem \ref{T1}.
To simplify the presentation, we will focus on the case $\nu=1$, that is, system \eqref{e1}.

We introduce the new variable:
$$\sigma:=\nabla\ln \r.$$
Then, we have

\begin{Lemma}\label{sigma}
Function $\sigma$ satisfies
\begin{equation}\label{xx2}
\partial_t \sigma+\nabla(\u\cdot\sigma)=0,
\end{equation}
in the sense of distributions. Moreover, the norm  $\|\sigma(t)\|_{L^q(\R^3)}$ is continuous in time.
\end{Lemma}

\begin{proof}
We follow the argument in \cite{AI} (Section 9.8) by denoting $\sigma_\varepsilon=S_\varepsilon\sigma$,
where $S_\varepsilon$ is the standard mollifier in the spatial variables. Then, we have
$$\partial_t\sigma_\varepsilon+\nabla(\u\cdot\sigma_\varepsilon)=\mathcal{R}_\varepsilon,$$
with
\begin{equation}\label{xxx1}
\begin{split}
\mathcal{R}_\varepsilon&=\nabla(\u\cdot\sigma_\varepsilon)-S_\varepsilon\nabla(\u\cdot\sigma)
=(\u\cdot\nabla\sigma_\varepsilon-S_\varepsilon(\u\cdot\nabla\sigma))+(\sigma_\varepsilon\nabla \u-S_\varepsilon(\sigma\cdot\nabla \u))\\
&=:\mathcal{R}_\varepsilon^1+\mathcal{R}_\varepsilon^2.
\end{split}
\end{equation}
Since $\sigma\in L^\infty(0,T; L^q(\R^3))$ and $\u\in L^p(0,T; W^{1,\infty}(\R^3))$, we deduce from Lemma 6.7 in \cite{AI}
(cf. Lemma 2.3 in \cite{LI}) that $\mathcal{R}_\varepsilon^1\to   0$ as $\varepsilon\to   0$.  Moreover,
$$\|(\sigma_\varepsilon-\sigma)\nabla \u\|_{L^1(0,T; L^q(\R^3))}\le\|\sigma-\sigma_\varepsilon\|_{L^{\f{p}{p-1}}(0,T; L^q(\R^3))}
\|\nabla \u\|_{L^p(0,T; L^\infty(\R^3))}\to   0,$$
and $S_\varepsilon(\sigma\cdot\nabla \u)\to  \sigma\cdot\nabla \u$ in $L^1(0,T; L^q(\R^3))$
since $\sigma\cdot\nabla \u\in L^p(0,T; L^q(\R^3))$.
Thus, we have $\mathcal{R}_\varepsilon^2\to   0$ in $L^p(0,T; L^q(\R^3))$.
Then, taking the limit as $\varepsilon\to   0$ in \eqref{xxx1}, we get \eqref{xx2}.

Multiplying \eqref{xx2} by $|\sigma|^{q-2} \sigma$, and
integrating over $\R^3$, we get
\begin{equation*}
\begin{split}
\f{1}{q}\left|\f{d}{dt}\|\sigma\|^q_{L^q(\R^3)}\right|&=\left|\int_{\R^3}(-\partial_j\u_k
\sigma_j\sigma_k|\sigma|^{q-2}
-\f{1}{q}\Dv \u|\sigma|^q)dx\right|\\
&\le \|\nabla \u\|_{L^\infty}\|\sigma\|^q_{L^q}+\f{1}{q}\|\Dv \u\|_{L^\infty}\|\sigma\|_{L^q}^q
\le C\|\u\|_{W^{2,q}}\|\sigma\|^{q}_{L^q}.
\end{split}
\end{equation*}
Dividing the above inequality by $\|\sigma\|^{q-1}_{L^q}$, we obtain
$$\left|\f{d}{dt}\|\sigma\|_{L^q}\right|\le C\|\u\|_{W^{2,q}}\|\sigma\|_{L^q}.$$
Since $\sigma\in L^\infty(0, T; L^q(\R^3))$, $\f{d}{dt}\|\sigma\|_{L^q}\in L^p(0,T)$. Thus, $\|\sigma\|_{L^q}\in C(0,T).$
The proof of Lemma \ref{sigma} is complete.
\end{proof}

For a given $R=\dl_0\ll 1$ as in Section \ref{local}, if the initial data satisfies $\|\u(0), \r_0-1, E(0)\|_{V_0}\le\dl^2$ with $0<\dl\ll
\min\{\f{1}{3}, \dl_0\}$,
let $T(R)$ be the maximal time $T$ such that there is a solution of
the equation $\u={\mathcal H}(\u)$ in $B_R(0)$. By virtue of Lemma \ref{r}, Lemma \ref{g} and Lemma \ref{sigma}, we know that
$\|\S(\u)-1\|_{W^{1,q}(\R^3)}$,
$\|\sigma\|_{L^q}$ and $\|\T(\u)\|_{W^{1,q}}$
are continuous in the interval $[0,T(R))$. On the other hand, under the assumptions on initial data and Remark \ref{yy1}, we know,  if $\dl$ is sufficiently small, then
$$\|\sigma(0)\|_{L^q(\R^3)}\le\f{1}{C_0}\|\nabla \r(0)\|_{L^q(\R^3)}\le\dl^{\f{3}{2}}\ll 1.$$
Hence, there exists a maximum positive number $T_1$ such that
\begin{equation}\label{xxxx1}
\max\big\{\|\S(\u)-1\|_{W^{1,q}}(t), \|\sigma\|_{L^{q}}(t), \|\T(\u)\|_{W^{1,q}}(t)\big\}\le\sqrt{R}\ll1\quad \textrm{for all}\quad t\in[0,T_1].
\end{equation}
Now, we denote $T=\min\{T(R), T_1\}$. Without loss of generality, we assume that $T<\infty$.

Since $q>3$, we have
$$\|\r-1\|_{L^\infty(\R^3)}\le C\|\r-1\|_{W^{1,q}(\R^3)}\le C\sqrt{R}<\f{1}{2},$$
if $R$ is sufficiently small.
Hence, one obtains
$$\f{1}{2}\le\r\le \f{3}{2}.$$
On the other hand, for any given $t\in(0,T)$, we can write
\begin{equation*}
\begin{split}
\|\u(t)\|_{L^q}^p&=\|\u(0)\|_{L^q}^p+\int_0^t\f{d}{ds}\|\u(s)\|_{L^q}^p ds\\
&=\|\u_0\|^p_{L^q}+\f{p}{q}\int_0^t \left(\|\u(t)\|_{L^q}^{p-q}\int_{\R^3} |\u(s)|^{q-2}\u(s)\partial_s\u(t)dx\right)dt\\
&\le\|\u_0\|^p_{L^q}+\f{p}{q}\int_0^t\|\u(s)\|_{L^q}^{p-1}\|\partial_s \u\|_{L^q}ds\\
&\le \dl^{2p}+\f{p}{q}\left(\int_0^t\|\u\|_{L^q}^pds\right)^{\f{p-1}{p}}\left(\int_0^t\|\partial_s \u\|
_{L^q}^pds\right)^{\f{1}{p}}\\
&\le\dl^{2p}+\f{p}{q}R^p,
\end{split}
\end{equation*}
and consequently,
\begin{equation}\label{21}
\|\u\|_{L^\infty(0,t; L^q)}\le \left(\dl^{2p}+\f{p}{q}R^p\right)^{\f{1}{p}}\le CR,\quad t\in(0,T).
\end{equation}
Similarly, we have, for all $t\in [0,T]$,
$\|\u\|_{L^\infty(0,t;L^2)}\le CR.$

\subsection{Dissipation of the deformation gradient}
The main difficulty of the proof of Theorem \ref{T1} is to obtain
estimates on the dissipation of the deformation gradient and the
gradient of the density. This is partly because of the transport
structure of equation \eqref{e13}. It is worthy of pointing out
that it is extremely difficult to directly deduce the dissipation
of the deformation gradient. Fortunately, for the viscoelastic
fluids system \eqref{e1}, as we can see in \cite{CZ, LLZH2, LLZH,
LLZH31, LZ, LLZ, LZP}, some sort of combinations between the
gradient of the velocity and the deformation gradient indeed
induce good dissipation. To make this statement more precise, we
rewritten the momentum equation \eqref{e12} as, using \eqref{e11}
\begin{equation}\label{511}
\begin{split}
\partial_t \u-\mu\Delta \u-(\mu+\lambda)\nabla\Dv\u-\Dv E&=-\r(\u\cdot\nabla)\u-\nabla P(\r)+\Dv(\r(I+E)^\top)\\
&\quad+\Dv((\r-1)E)+\Dv(\r EE^\top)+(1-\r)\partial_t\u,
\end{split}
\end{equation}
and prove the following estimate:

\begin{Lemma}\label{E1}
\begin{equation}\label{5113}
\|\nabla E\|_{L^p(0,T; L^q(\R^3))}
\le C(p,q,\mu)\left(R+\sqrt{R}\|\sigma\|_{L^p(0,T; L^q(\R^3))}\right).
\end{equation}
\end{Lemma}

\begin{proof}
Now we introduce an operator $\mathcal{L}: \omega\mapsto
\mathcal{L}\omega$ such that $$-\mu
\D\mathcal{L}(\omega)-(\mu+\lambda)\nabla\Dv\mathcal{L}(\omega)=\omega.$$
Then we denote $Z_1(x,t)$ as
\begin{equation}\label{512}
Z_1:=\mathcal{L}(\Dv E).
\end{equation}
Notice that, from the elliptic theory, if $E\in L^p(\R^3)$ with
$1<p<\infty$, then $Z_1\in W^{1,p}(\R^3)$.
%
Then, \eqref{511} becomes
\begin{equation}\label{513}
\partial_t \u-\mu\Delta\left(\u-Z_1\right)-(\mu+\lambda)\nabla\Dv(\u-Z_1)=\mathcal{F}_1,
\end{equation}
where, with the help of Remark \ref{QQ},
\begin{equation*}
\begin{split}
\mathcal{F}_1=-\r(\u\cdot\nabla)\u-\nabla
P(\r)+\Dv((\r-1)E)+\Dv(\r EE^\top)+(1-\r)\partial_t\u.
\end{split}
\end{equation*}
Also, from \eqref{e1}, we have
\begin{equation}\label{514}
\begin{split}
\f{\partial Z_1}{\partial t}&=\mathcal{L}\left(\f{\partial (\Dv E)}{\partial t}\right)=\mathcal{L}\left(\Dv\left(\f{\partial E}{\partial t}\right)\right)=\mathcal{L}\left(\Dv (\nabla \u+\nabla \u
E-(\u\cdot\nabla)E)\right).
\end{split}
\end{equation}
From \eqref{513} and \eqref{514}, we deduce, denoting $Z=\u-Z_1$,
\begin{equation}\label{515}
\begin{split}
\partial_t Z-\mu\Delta Z-(\mu+\lambda)\nabla\Dv Z=\mathcal{F}:=\mathcal{F}_1-\mathcal{F}_2,
\end{split}
\end{equation}
where
$\mathcal{F}_2=\mathcal{L}\left(\Dv (\nabla \u+\nabla \u
E-(\u\cdot\nabla)E)\right).$
Equation \eqref{515} with Theorem \ref{T3} implies
\begin{equation}\label{516}
\begin{split}
\|Z\|_{\mathcal{W}(0,T)}&\le C(p,q,\mu,\lambda)\left(\|Z(0)\|_{X_p^{2(1-\f{1}{p})}}+\|\mathcal{F}\|_{L^p(0,T;L^q(\R^3))}\right)\\
&\le
C(p,q,\mu,\lambda)\left(R+\|\mathcal{F}\|_{L^p(0,T;L^q(\R^3))}\right).
\end{split}
\end{equation}

Next, we estimate $\|\mathcal{F}_i\|_{L^p(0,T;L^q(\R^3))}$, $i=1,2$. Indeed, for $\mathcal{F}_1$, using \eqref{21}, we
have
\begin{equation}\label{517}
\begin{split}
&\|\mathcal{F}_1\|_{L^p(0,T; L^q(\R^3))}\\
&\le \|\r\|_{L^\infty(Q_T)}\|\u\|_{L^\infty(0,T; L^q(\R^3))}\|\nabla \u\|_{L^p(0,T; L^\infty(\R^3))}\\
&\quad+\alpha\|\nabla E\|_{L^p(0,T; L^q(\O))}+\alpha\|\sigma\|_{L^p(0,T; L^q(\O))}\|E\|_{L^\infty(Q_T)}\\
&\quad+\|\r-1\|_{L^\infty(Q_T)}\|\nabla E\|_{L^p(0,T; L^q(\R^3))}\\
&\quad+\|\sigma\|_{L^p(0,T; L^q(\R^3))}\|E\|_{L^\infty(Q_T)}+\|\sigma\|_{L^p(0,T; L^q(\R^3))}\|E\|^2_{L^\infty(Q_T)}\\
&\quad+\|E\|_{L^\infty(Q_T)}\|\nabla E\|_{L^p(0,T; L^q(\R^3))} +\|\r-1\|_{L^\infty(Q_T)}\|\partial_tv\|_{L^p(0,T; L^q(\R^3))}\\
&\le 2R\|\u\|_{L^p(0,T; W^{2,q}(\R^3))}+\alpha\|\nabla E\|_{L^p(0,T; L^q(\R^3))}\\
&\quad+\alpha\|\sigma\|_{L^p(0,T; L^q(\R^3))}\|E\|_{L^\infty(Q_T)}
+\sqrt{R}\|\sigma\|_{L^p(0,T; L^q(\R^3))}\\
&\quad+R\|\sigma\|_{L^p(0,T; L^q(\R^3))}
+\sqrt{R}\|\nabla E\|_{L^p(0,T; L^q(\R^3))}+R^{\f{3}{2}}\\
&\le R^{\f{3}{2}}+\alpha\|\nabla E\|_{L^p(0,T; L^q(\R^3))}+\sqrt{R}\|\sigma\|_{L^p(0,T; L^q(\R^3))} +\sqrt{R}\|\nabla E\|_{L^p(0,T; L^q(\R^3))}.
\end{split}
\end{equation}
Here, $\alpha=\sup\big\{xP'(x):\f{1}{2}\le x\le \f{3}{2}\big\}$ and in the first inequality, we used the identity
$$\nabla\r=-\r\Dv E^\top-\nabla\r E^\top$$ due to Remark \ref{QQ}.
And, for $\mathcal{F}_2$, we have
\begin{equation*}
\begin{split}
&\|\nabla\u+\nabla \u E-(\u\cdot\nabla)E\|_{L^p(0,T; L^{\f{3q}{q+3}}(\R^3))}\\
&\le \|\u\|_{L^p(0,T; W^{2,q}\cap H^2)}
+\|\nabla\u\|_{L^p(0,T; L^3)}\|E\|_{L^\infty(0,T; L^q)}
+\|\u\|_{L^p(0,T; L^3)}\|\nabla E\|_{L^\infty(0,T; L^q)}\\
&\le R+R^{\f{3}{2}}.
\end{split}
\end{equation*}
Here, we used the following Gagliardo-Nirenberg inequality
$$\|\nabla\u\|_{L^p(0,T; L^{\f{3q}{q+3}}(\R^3))}\le
\left(\int_0^T\left(\|\u(s)\|^\theta\|\D\u(s)\|_{L^q}^{1-\theta}\right)^pds\right)^{\f{1}{p}}\le\|\u\|_{L^p(0,T; W^{2,q}\cap H^2)},
$$ with $\theta=\f{4q}{7q-6}.$
Hence, one can estimate, by $L^p$ estimates of elliptic operators,
\begin{equation}\label{518}
\begin{split}
\|\mathcal{F}_2\|_{L^p(0,T; L^{q}(\R^3))}
&\le\|\mathcal{F}_2\|_{L^p(0,T; W^{1,\f{3q}{3+q}})}\\
&\le C(\mu,\lambda)\left\|\nabla \u+\nabla \u
E-(\u\cdot\nabla)E\right\|_{L^p(0,T; L^{\f{3q}{3+q}})}\\
&\le C(\mu, \lambda)(R+R^{\f{3}{2}}).
\end{split}
\end{equation}
Therefore, from \eqref{517} and \eqref{518}, we obtain
\begin{equation}\label{519}
\begin{split}
\|\mathcal{F}\|_{L^p(0,T; L^q(\R^3))}&\le C(\mu, \lambda)(
R^{\f{3}{2}}+R)+\alpha\|\nabla E\|_{L^p(0,T; L^q)}+\sqrt{R}
\|\sigma\|_{L^p(0,T; L^q)}\\&\quad+\sqrt{R}\|\nabla E\|
_{L^p(0,T; L^q)}.
\end{split}
\end{equation}
Inequalities \eqref{516} and \eqref{519} imply that
\begin{equation}\label{5110}
\begin{split}
&\|Z\|_{L^p(0,T; W^{2,q}(\R^3))}\\
&\le C(p,q, \mu, \lambda)\Big(R+\alpha\|\nabla E\|_{L^p(0,T;
L^q)}+\sqrt{R}\|\sigma\|_{L^p(0,T; L^q)}
+\sqrt{R}\|\nabla E\|_{L^p(0,T; L^q)}\Big).
\end{split}
\end{equation}
Hence, we have, from \eqref{512}
\begin{equation}\label{5111}
\begin{split}
&\|\Dv E\|_{L^p(0,T; L^q(\R^3))}\le \mu\|Z\|_{L^p(0,T; W^{2,q}(\R^3))}+\|\u\|_{L^p(0,T; W^{2,q}(\R^3))}\\
&\le C(p,q, \mu, \lambda)\Big(R+\sqrt{R} \|\sigma\|_{L^p(0,T;
L^q(\R^3))}+\sqrt{R}\|\nabla E\| _{L^p(0,T;
L^q(\R^3))}\Big)\\&\quad+C(p,q, \mu, \lambda)\mu\alpha\|\nabla
E\|_{L^p(0,T; L^q(\R^3))}.
\end{split}
\end{equation}
On the other hand, from the identity \eqref{11112}, we deduce that
\begin{equation}\label{5112}
\begin{split}
\|\textrm{curl }E_i\|_{L^p(0,T; L^q(\R^3))}&\le 2\|E\|_{L^\infty(Q_T)}\|\nabla E\|_{L^p(0,T; L^q(\R^3))}\\
&\le C\|E\|_{L^\infty(0,T; W^{1,q})}\|\nabla E\|_{L^p(0,T; L^q)}
\le C\sqrt{R}\|\nabla E\|_{L^p(0,T; L^q)}.
\end{split}
\end{equation}
Combining together \eqref{5111} and \eqref{5112}, we obtain
\begin{equation*}
\begin{split}
\|\nabla E\|_{L^p(0,T; L^q(\R^3))} &\le C(p,q, \mu, \lambda)\Big(R+\sqrt{R} \|\sigma\|_{L^p(0,T; L^q)}+\sqrt{R}\|\nabla
E\| _{L^p(0,T; L^q)}\Big)\\&\quad+C(p,q, \mu,\lambda)\mu\alpha\|\nabla E\|_{L^p(0,T; L^q])},
\end{split}
\end{equation*}
and hence, by choosing $\sqrt{R}\ll \f{1}{2}$ and using the
assumption
\begin{equation}\label{AA}
C(p,q,\mu,\lambda)\mu\alpha<1,
\end{equation}
one obtains \eqref{5113}. The proof of Lemma \ref{E1} is complete.
\end{proof}

\begin{Remark}
The assumption \eqref{AA} is reasonable, because if we consider
the special case: $0<\mu\le 1$ and $\lambda=0$. Then, after the
scaling, we get a control on the constant $C(p,q,\mu,\lambda)\le
C(p,q)\mu^{-\f{6+q}{3q}}$, and hence
$C(p,q,\mu,\lambda)\mu\alpha\le \alpha\mu^{\f{2q-6}{3q}}\rightarrow
0,\quad\textrm{as}\quad \mu\rightarrow 0.$
\end{Remark}

\begin{Remark}
Notice that, in view of the above argument,  estimate \eqref{5113} is actually valid for all $t\in[0,T]$, that is, for all $t\in[0,T]$,
\begin{equation*}
\begin{split}
\|\nabla E\|_{L^p(0,t; L^q(\R^3))} &\le
C(p,q,\mu,\lambda)\left(R+\sqrt{R}\|\sigma\|_{L^p(0,t;
L^q(\R^3))}\right).
\end{split}
\end{equation*}
\end{Remark}

\subsection{Dissipation of the gradient of the density}
To make Theorem \ref{T1} valid, we need further  the uniform estimate on the dissipation of the gradient of the density.

\begin{Lemma}\label{E2}
For any $t\in (0, T)$,
\begin{equation}\label{12345}
\|\sigma\|_{L^p(0,t; L^q(\R^3))}\le C(p,q,\mu)R.
\end{equation}
\end{Lemma}

\begin{proof}
Multiplying \eqref{e12} by $\sigma|\sigma|^{q-2}$ and integrating over $\R^3$, we obtain
\begin{equation}\label{17}
\begin{split}
&\f{\mu+\lambda}{q}\f{d}{dt}\|\sigma\|^q_{L^q}+\int_{\R^3} \r P'(\r)|\sigma|^qdx\\
&=\int_{\R^3}(\mu\D \u+(\mu+\lambda)\nabla\Dv\u)\cdot\sigma|\sigma|^{q-2}dx-\int_{\R^3} \r\partial_t \u\cdot\sigma|\sigma|^{q-2}dx\\
&\quad-\int_{\R^3} \r(\u\cdot\nabla)\u\cdot\sigma|\sigma|^{q-2}dx-(\mu+\lambda)\int_{\R^3}\nabla(\u\cdot\sigma)\cdot\sigma|\sigma|^{q-2}dx\\
&\quad+\int_{\R^3}
\Dv(\r(I+E)(I+E)^\top)\cdot\sigma|\sigma|^{q-2}dx.
\end{split}
\end{equation}

We estimate the right-hand side of \eqref{17} term by term,
$$\left|\int_{\R^3} (\mu\D \u+(\mu+\lambda)\nabla\Dv\u)\cdot\sigma|\sigma|^{q-2}dx\right|\le\|\u\|_{W^{2,q}}\|\sigma\|_{L^q}^{q-1};$$
$$\left|\int_{\R^3} \r\partial_t \u\cdot\sigma|\sigma|^{q-2}dx\right|\le \|\partial_t \u\|_{L^q}\|\sigma\|_{L^q}^{q-1};$$
\begin{equation*}
\begin{split}
\left|\int_{\R^3} \r\u\cdot\nabla \u\cdot\sigma|\sigma|^{q-2}dx\right|&\le \|\u\cdot\nabla \u\|_{L^q}\|\sigma\|^{q-1}_{L^q}\
\le \|\u\|_{L^q}\|\u\|_{W^{2,q}}\|\sigma\|_{L^q}^{q-1};
\end{split}
\end{equation*}
\begin{equation*}
\begin{split}
\left|\int_{\R^3} \nabla(\u\cdot\sigma)\cdot\sigma|\sigma|^{q-2}dx\right|&=\left|\int_{\R^3}\partial_j\u_k\sigma_k\sigma_j|
\sigma|^{q-2}dx+\int_{\R^3} \u_k\partial_j\partial_k(\ln \r)\partial_j\ln \r|\sigma|^{q-2}dx\right|\\
&=\left|\int_{\R^3}\partial_j\u_k\sigma_k\sigma_j|\sigma|^{q-2}dx+\f{1}{2}\int_{\R^3} \u_k\partial_k|\sigma|^2|\sigma|^{q-2}dx\right|\\
&=\left|\int_{\R^3}\partial_j\u_k\sigma_k\sigma_j|\sigma|^{q-2}dx+\int_{\R^3} \u_k\partial_k|\sigma||\sigma|^{q-1}dx\right|\\
&=\left|\int_{\R^3}\partial_j\u_k\sigma_k\sigma_j|\sigma|^{q-2}dx+\f{1}{q}\int_{\R^3} \u_k\partial_k|\sigma|^qdx\right|\\
&=\left|\int_{\R^3}\partial_j\u_k\sigma_k\sigma_j|\sigma|^{q-2}dx-\f{1}{q}\int_{\R^3} |\sigma|^q\Dv \u dx\right|\\
&\le C\|\nabla \u\|_{L^\infty}\|\sigma\|^q_{L^q}\le C\|\u\|_{W^{2,q}}\|\sigma\|_{L^q}^q,
\end{split}
\end{equation*}
and, due to \eqref{cc1}, we can rewrite
\begin{equation*}
\begin{split}
\left(\Dv(\r(I+E)(I+E)^\top)\right)_i
&=\f{\partial{(\r(e_i+E_i)(e_j+E_j))}}{\partial x_j}\\&=(e_i+E_i)\f{\partial{(\r(e_j+E_j))}}{\partial x_j}
+\r(e_j+E_j)\f{\partial{(e_i+E_i)}}{\partial x_j}\\
&=\r(e_j+E_j)\f{\partial{E_i}}{\partial x_j},
\end{split}
\end{equation*}
then one has
\begin{equation*}
\begin{split}
&\left|\int_{\R^3}\Dv(\r(I+E)(I+E)^\top)\cdot\sigma|\sigma|^{q-2}dx\right|=\left|\int_{\R^3} \r(e_j+E_j)\f{\partial{E_i}}{\partial x_j}
\sigma_i|\sigma|^{q-2}dx\right|\\
&\le \|\nabla E\|_{L^q}\|I+E\|_{L^\infty}\|\sigma\|^{q-1}_{L^q}
\le 2\|\nabla E\|_{L^q(\R^3)}\|\sigma\|^{q-1}_{L^q}.
\end{split}
\end{equation*}

On the other hand, we have
\begin{equation*}
\begin{split}
\r P'(\r)&=P'(1)+(\r-1)\int_0^1 \left(P'(\eta (\r-1)+1) +(\eta (\r-1)+1)P''(\eta (\r-1)+1)\right)d\eta,
\end{split}
\end{equation*}
for $\r$ close to 1 and consequently
\begin{equation*}
\begin{split}
\|\r P'(\r)-P'(1)\|_{L^\infty}
&\le\|\r-1\|_{L^\infty}\sup\left\{|f(x)|:\f{1}{2}\le x\le \f{3}{2}\right\}\\
&\le C\sqrt{R}\sup\left\{|f(x)|:\f{1}{2}\le x\le \f{3}{2}\right\}\le C\sqrt{R},
\end{split}
\end{equation*}
where $f(x)=P'(x)+xP''(x)$.
Thus, from \eqref{17}, we obtain
\begin{equation*}
\begin{split}
&\f{\mu+\lambda}{q}\f{d}{dt}\|\sigma\|^q_{L^q}+P'(1)\|\sigma\|^q_{L^q}\\
&\le\|\sigma\|^{q-1}_{L^q}\Big(\|\D \u\|_{L^q}+\|\partial_t \u\|_{L^q} +\|\u\|_{W^{2,q}}\|\u\|_{L^q}\\
&\qquad\qquad\quad +C\|\u\|_{W^{2,q}}\|\sigma\|_{L^q}+C\|\nabla E\|_{L^q(\R^3)}+\sqrt{R}\|\sigma\|_{L^q}\Big)\\
&\le C\|\sigma\|_{L^q}^{q-1}\Big(\|\D \u\|_{L^q}+\|\partial_t \u\|_{L^q}+\|\u\|_{W^{2,q}}\|\u\|_{L^q}\\
&\qquad\qquad\qquad +\|\u\|_{W^{2,q}}\|\sigma\|_{L^q}+\|\nabla E\|_{L^q(\R^3)}+\sqrt{R}\|\sigma\|_{L^q}\Big),
\end{split}
\end{equation*}
and hence, by assuming that $R\ll 1$, one obtains
\begin{equation}\label{19}
\begin{split}
&\f{\mu+\lambda}{q}\f{d}{dt}\|\sigma\|^q_{L^q}+\f{1}{2}P'(1)\|\sigma\|^q_{L^q}\\
&\le C\|\sigma\|_{L^q}^{q-1}\big(\|\D \u\|_{L^q}+\|\partial_t \u\|_{L^q}+\|\u\|_{W^{2,q}}\|\u\|_{L^q}
+\|\u\|_{W^{2,q}}\|\sigma\|_{L^q}+\|\nabla E\|_{L^q}\big).
\end{split}
\end{equation}
Multiplying \eqref{19} by $\|\sigma\|^{p-q}_{L^q}$, we obtain
\begin{equation*}
\begin{split}
&\f{\mu+\lambda}{p}\f{d}{dt}\|\sigma\|^p_{L^q}+\f{1}{2}P'(1)\|\sigma\|^p_{L^q}\\
&\le C\|\sigma\|^{p-1}_{L^q}\big(\|\D \u\|_{L^q}+\|\partial_t \u\|_{L^q}+\|\u\|_{W^{2,q}}\|\u\|_{L^q}
+\|\u\|_{W^{2,q}}\|\sigma\|_{L^q}+\|\nabla E\|_{L^q}\big).
\end{split}
\end{equation*}
Integrating the above inequality over the interval $(0, t)$, one obtains, by using \eqref{5113},
\begin{equation*}
\begin{split}
&\f{\mu+\lambda}{p}\|\sigma(t)\|^p_{L^q}+\f{1}{2}P'(1)\int_0^t\|\sigma\|^p_{L^q}ds\\
&\le\f{\mu+\lambda}{p}\|\sigma(0)\|^p_{L^q}
+C\left(\int_0^t\|\sigma\|_{L^q}^pds\right)^{\f{p-1}{p}}\left(\left(\int_0^t\|\partial_t\u\|^p_{L^q}ds\right)^{\f{1}{p}}\right.\\
&\qquad +\left(\|\u\|_{L^\infty(0,t;L^q(\R^3))}+\|\sigma\|_{L^\infty(0,t;L^q(\R^3))}+1\right)
\left(\int_0^t\|\u\|_{W^{2,q}}^pds\right)^{\f{1}{p}}\\
&\qquad +\|\nabla E\|_{L^p(0,t;L^q(\R^3))}\Big)\\
&\le\f{\mu+\lambda}{p}\|\sigma(0)\|^p_{L^q}+C(p,q,\mu)\sqrt{R}\|\sigma\|_{L^p(0,T; L^q(\R^3))}^p\\
&\qquad+C(p,q,\mu)R\left(\int_0^t\|\sigma\|_{L^q}^pds\right)^{\f{p-1}{p}}
\left(1+\|\sigma\|_{L^\infty(0,t;
L^q(\R^3))}+\|\u\|_{L^\infty(0,t; L^q(\R^3))}\right).
\end{split}
\end{equation*}
and hence, by letting $R$ be so small such that $C(p,q,\mu)\sqrt{R}<\f{1}{4}$, one obtains
\begin{equation}\label{20}
\begin{split}
&\f{\mu+\lambda}{p}\|\sigma(t)\|^p_{L^q}+\f{1}{4}P'(1)\int_0^t\|\sigma\|^p_{L^q}ds\\
&\le\f{\mu+\lambda}{p}\|\sigma(0)\|^p_{L^q}
+C(p,q,\mu)R\left(\int_0^t\|\sigma\|_{L^q}^pds\right)^{\f{p-1}{p}}\left(1+\|\sigma\|_{L^\infty(0,t;
L^q)} +\|\u\|_{L^\infty(0,t; L^q)}\right).
\end{split}
\end{equation}

Plugging \eqref{21} into \eqref{20}, we obtain
\begin{equation*}
\begin{split}
&\f{\mu+\lambda}{p}\|\sigma(t)\|^p_{L^q}+\f{1}{4}P'(1)\int_0^t\|\sigma\|^p_{L^q}ds\\&\le\f{\mu+\lambda}{p}\|\sigma(0)\|^p_{L^q}
+C(p,q,\mu)R\left(\int_0^t\|\sigma\|_{L^q}^pds\right)^{\f{p-1}{p}}(1+\|\sigma\|
_{L^\infty(0,t; L^q(\R^3))}).
\end{split}
\end{equation*}
Then, Young's inequality yields
\begin{equation}\label{22}
\begin{split}
&\f{\mu+\lambda}{p}\|\sigma(t)\|^p_{L^q}+\f{1}{8}P'(1)\int_0^t\|\sigma\|^p_{L^q}ds\\
&\le \f{\mu+\lambda}{p}\dl^{\f{3}{2}p}+C(p,q,\mu)R^p(1+\|\sigma\|_{L^\infty(0,t; L^q)})^p,
\end{split}
\end{equation}
for all $0\le t<T$.

Now, we let $R$ be so small that
$$C(p,q,\mu)^{\f{1}{p}}\sqrt{R}\left(1+\sqrt{R}\right)<\f{1}{2}.$$
Due to the fact that $\|\sigma(0)\|_{L^q(\R^3)}\le \dl^{\f{3}{2}}$, we can assume that $\|\sigma(t)\|_{L^q}<\f{1}{2}\sqrt{R}$
in some maximal
interval $(0, t_{\textrm{max}})\subset(0,T).$ If $t_{\textrm{max}}<T$, then, $\|\sigma(t_{\textrm{max}})\|_{L^q}=\f{1}{2}\sqrt{R}$ and by \eqref{22},
$$\f{1}{2}\sqrt{R}=\|\sigma(t_{\textrm{max}})\|_{L^q}\le C(p,q,\mu)^{\f{1}{p}}R(1+\sqrt{R})<\f{1}{2}\sqrt{R},$$
which is a contradiction. Hence, $t_{\textrm{max}}=T$ and
\begin{equation}\label{eo}
\|\sigma\|_{L^q}\le \f{1}{2}\sqrt{R},\quad\textrm{for all}\quad t\in[0,T].
\end{equation}
Thus, by \eqref{22}, one obtains \eqref{12345}. The proof of Lemma \ref{E2} is complete.
\end{proof}

We remark that,  from \eqref{5113} and \eqref{12345}, one has
\begin{equation}\label{2222}
\|\nabla E\|_{L^p(0,T; L^{q}(\R^3))}\le C(p,q,\mu)R.
\end{equation}

\section{Refined Uniform Estimates} \label{global}

In this section, we prove the second part and thus complete the proof of Theorem \ref{T1}. Define
\begin{equation*}
\begin{split}
T_{\textrm{max}}:=\sup\Big\{T>0: \;\; &\exists \; \u\in \mathcal{W}(0,T)\;  \textrm{with } \u={\mathcal H}(\u),\:
\textrm{such that,}\;\;
\|\u\|_{\mathcal{W}(0,T)}\le R, \\ 
&\|\S(\u)-1\|_{L^\infty(0,T;W^{1,q})}\le\sqrt{R}, \; \|\sigma\|_{L^\infty(0,T;L^q)}\le\sqrt{R}, \text{ and }\\
&\|\T(\u)\|_{L^\infty(0,T;W^{1,q})}\le\sqrt{R}
\Big\},
\end{split}
\end{equation*}
where $R$ was constructed in the previous section.

\subsection{Uniform estimates in time}

We now establish some estimates which are uniform in time $T$.  First we prove  the following energy estimates:

\begin{Lemma}\label{ul}
Under the same assumptions as Theorem \ref{T20}, we have
\begin{equation}\label{ul1}
\|\nabla\u\|_{L^2(0,T; L^2(\R^3))}\le CR^2,
\end{equation}
\begin{equation}\label{ul123}
\|\u\|_{L^\infty(0,T; L^2(\R^3))}\le CR^2,
\end{equation}
\begin{equation}\label{w7}
\|E\|_{L^\infty(0,T; L^2(\R^3))}\le CR^2,
\end{equation}
\begin{equation}\label{w8}
\|\r-1\|_{L^\infty(0,T; L^2(\R^3))}\le CR^2,
\end{equation}
where $C$ is a constant independent of $T\in(0, T_{\rm{max}})$.
\end{Lemma}
\begin{proof}
First we recall that $$\u\in W^{1,2}(0,T; L^2(\R^3))\cap L^2(0,T; W^{2,2}(\R^3))$$
and
$$\r, E\in W^{1,2}(0,T;L^2(\R^3))
\cap L^2(0,T; W^{1,2}(\O)).$$
Multiplying  equation \eqref{e12} by $\u$, and integrating over $\R^3$, we obtain, using the conservation of mass \eqref{e11},
\begin{equation*}
\begin{split}
&\f{d}{dt}\int_{\R^3}\left(\f{1}{2}\r|\u|^2+\f{1}{\gamma-1}(\r^\gamma+\gamma-1)\right)dx+\int_{\R^3}(\mu|\nabla\u|^2+(\mu+\lambda)|\Dv\u|^2)dx\\
&=-\int_{\R^3}\r \F\F^\top:\nabla\u dx.
\end{split}
\end{equation*}
Here, the notation $A:B$ means the dot product between two
matrices. Thus, we have
\begin{equation*}
\begin{split}
&\int_{\R^3}\left(\f{1}{2}\r|\u|^2+\f{1}{\gamma-1}(\r^\gamma+\gamma-1)\right)dx+\int_0^t\int_{\R^3}(\mu|\nabla\u|^2+(\mu+\lambda)|\Dv\u|^2)dxds\\
&=\int_{\R^3}\left(\f{1}{2}\r_0|\u_0|^2+\f{1}{\gamma-1}(\r_0^\gamma+\gamma-1)\right)dx-\int_0^t\int_{\R^3}
\r \F\F^\top:\nabla\u dxds.
\end{split}
\end{equation*}
From the conservation of mass \eqref{e11}, one has
\begin{equation}\label{w1}
\begin{split}
&\int_{\R^3}\left(\f{1}{2}\r|\u|^2+\f{1}{\gamma-1}(\r^\gamma-\gamma
\r+\gamma-1)\right)dx+\int_0^t\int_{\R^3}(\mu|\nabla\u|^2+(\mu+\lambda)|\Dv\u|^2)dxds\\
&=\int_{\R^3}\left(\f{1}{2}\r_0|\u_0|^2+\f{1}{\gamma-1}(\r_0^\gamma-\gamma
\r_0+\gamma-1)\right)dx-\int_0^t\int_{\R^3}\r \F\F^\top:\nabla\u dxds.
\end{split}
\end{equation}
 On the other hand, due to equations \eqref{e13} and \eqref{e11}, we have
\begin{equation}\label{w2}
\begin{split}
\f{\partial}{\partial t}\left(\r |\F|^2\right)&=\f{\partial \r}{\partial t}|\F|^2+2\r \F:\f{\partial \F}{\partial t}
=\f{\partial \r}{\partial t}|\F|^2+2\r \F:(\nabla\u \,\F-\u\cdot\nabla \F)\\
&=\f{\partial \r}{\partial t}|\F|^2+2\r \F:(\nabla\u \,\F)-\r\u\cdot\nabla |\F|^2\\
&=\f{\partial \r}{\partial t}|\F|^2+2\r \F:(\nabla\u \,\F)+\Dv(\r\u)|\F|^2-\Dv(\r\u|\F|^2)\\
&=2\r \F:(\nabla\u \,\F)-\Dv(\r\u|\F|^2).
\end{split}
\end{equation}
Integrating \eqref{w2} over $\R^3$, we arrive at
\begin{equation}\label{w4}
\begin{split}
\f{1}{2}\f{d}{dt}\int_{\R^3}\r|\F|^2dx=\int_{\R^3}\r \F:(\nabla\u \,\F)dx.
\end{split}
\end{equation}
Since
$$\int_0^t\int_{\R^3}\r \F:(\nabla\u \,\F)dxds=\int_0^t\int_{\R^3}\r \F\F^\top:\nabla\u dxds,$$
we finally obtain, by summing \eqref{w1} and \eqref{w2},
\begin{equation}\label{w3}
\begin{split}
&\int_{\R^3}\left(\f{1}{2}\r|\u|^2+\f{1}{2}\r|\F|^2+\f{1}{\gamma-1}(\r^\gamma-\gamma
\r+\gamma-1)\right)dx\\&\quad+\int_0^t\int_{\R^3}(\mu|\nabla\u|^2+(\mu+\lambda)|\Dv\u|^2)dxds\\
&=\int_{\R^3}\left(\f{1}{2}\r_0|\u_0|^2+\f{1}{2}\r_0|\F_0|^2+\f{1}{\gamma-1}(\r_0^\gamma-\gamma
\r_0+\gamma-1)\right)dx.
\end{split}
\end{equation}

From Remark \ref{QQ},
$\r(I+E^\top):\nabla\u=0.$
Hence, from \eqref{e13} and \eqref{e11}, we have
\begin{equation}\label{w6}
\partial_t(\r \, \tr E)=0.
\end{equation}
Therefore, from \eqref{w3}, \eqref{w6} and the conservation of mass
\eqref{e11}, we finally arrive at
\begin{equation}\label{w5}
\begin{split}
&\int_{\R^3}\left(\f{1}{2}\r|\u|^2+\f{1}{2}\r|E|^2+\f{1}{\gamma-1}(\r^\gamma-\gamma
\r+\gamma-1)\right)dx\\&\quad+\int_0^t\int_{\R^3}(\mu|\nabla\u|^2+(\mu+\lambda)|\Dv\u|^2)dxds\\
&=\int_{\R^3}\left(\f{1}{2}\r_0|\u_0|^2+\f{1}{2}\r_0|E_0|^2+\f{1}{\gamma-1}(\r_0^\gamma-\gamma
\r_0+\gamma-1)\right)dx\le R^4.
\end{split}
\end{equation}
Since $\mu>0$ is a constant and $\r\in[\f{1}{2}, \f{3}{2}]$,   then inequalities \eqref{ul1}-\eqref{w7} follow from \eqref{w5}, and inequality \eqref{w8} follows from \eqref{w5} and the
following straightforward inequalities: for some $\eta>0$, we have
\begin{equation*}
x^\gamma-1-\gamma(x-1)\ge
\begin{cases}
\eta|x-1|^2,&\textrm{if}\quad \gamma\ge 2,\\
\eta|x-1|^2,&\textrm{if}\quad |x|<2\quad\textrm{and}\quad 1<\gamma<2.
\end{cases}
\end{equation*}
The proof of Lemma \ref{ul} is complete.
\end{proof}

Based on the uniform estimates from Section \ref{estimates}, we have

\begin{Lemma}\label{ue}
Under the same assumptions as Theorem \ref{T1}, 
\begin{equation}\label{ue1}
\|\S(\u)-1\|_{L^\infty(0,T; W^{1,q})}<\sqrt{R}, \quad \|\sigma\|_{L^\infty(0,T; L^{q})}<\sqrt{R},
\end{equation}
for any $T\in[0, T_{\rm{max}}]$.
\end{Lemma}
\begin{proof}
According to \eqref{eo}, it is obvious to see that
$$\max_{t\in[0,T]}\|\sigma\|_{L^{q}}(t)<\sqrt{R}.$$
Hence, we are only left to show
$$\max_{t\in[0,T]}\|\S(\u)-1\|_{W^{1,q}}(t)<\sqrt{R}.$$
Indeed, for any $t\in(0,T)$, we have, by using \eqref{e11} and \eqref{12345},
\begin{equation}\label{w9}
\begin{split}
&\|\S(\u)(t)-1\|_{L^q}^{\alpha}\\
&=\|\r_0-1\|_{L^q}^{\alpha}+\int_0^t\f{d}{ds}\|\S(\u)(s)-1\|_{L^q}^{\alpha} ds\\
&=\|\r_0-1\|^{\alpha}_{L^q}+\f{\alpha}{q}\int_0^t \Big(\|\S(\u)(s)-1\|_{L^q}^{\alpha-q}\\
&\qquad\qquad\qquad\qquad\qquad \quad \times \int_{\R^3} |\S(\u)(s)-1|^{q-2}(\S(\u)(s)-1)\partial_s\S(\u)(s)dx\Big)ds\\
&\le\|\r_0-1\|^{\alpha}_{L^q}+\f{\alpha}{q}\int_0^t\|\S(\u)(s)-1\|_{L^q}^{\alpha-1}\|\partial_s \S(\u)\|_{L^q}ds\\
&\le \dl^{2\alpha}+\f{\alpha}{q}\left(\int_0^t\|\S(\u)(s)-1\|_{L^q}^{\f{(5q-6)p}{3q-6}}ds\right)
^{\f{p-1}{p}}\left(\int_0^t\|\partial_s \S(\u)\|_{L^q}^pds
\right)^{\f{1}{p}},
\end{split}
\end{equation}
where $$\alpha=\f{(5q-6)(p-1)}{3q-6}+1.$$

From \eqref{e11}, we obtain
\begin{equation}\label{w10}
\begin{split}
\|\partial_t \r\|_{L^p(0,T; L^q(\R^3))}&=\|\nabla\r\cdot\u\|_{L^p(0,T; L^q(\R^3))}+\|\r\Dv\u\|_{L^p(0,T; L^q(\R^3))}\\
&\le 2\|\sigma\|_{L^\infty(0,T; L^q(\R^3))}\|\u\|_{L^p(0,T;
L^\infty(\R^3))}+2\|\Dv\u\|_{L^p(0,T; L^q(\R^3))}\\&\le CR^2+R\le
CR .
\end{split}
\end{equation}

On the other hand, from the Gagliardo-Nirenberg inequality, we have
$$\|\r-1\|_{L^q(\R^3)}\le C\|\r-1\|_{L^2(\R^3)}^\theta\|\nabla(\r-1)\|_{L^q(\R^3)}^{1-\theta}\le C\|\r-1\|_{L^2(\R^3)}^\theta
\|\sigma\|_{L^q(\R^3)}^{1-\theta},$$
with $\theta=\f{2q}{5q-6}$. Thus, by H\"{o}lder's inequality, \eqref{w8}, and \eqref{12345}, one has
\begin{equation*}
\begin{split}
\|\r-1\|_{L^{\f{(5q-6)p}{3q-6}}(0,T;L^q(\R^3))}\le C\|\r-1\|_{L^\infty(0,T; L^2(\R^3))}^\theta
\|\sigma\|_{L^p(0,T; L^q(\R^3))}^{1-\theta}\le CR,
\end{split}
\end{equation*}
which, together \eqref{w9} and \eqref{w10}, yields
$$\|\S(\u)(t)-1\|_{L^q}\le CR.$$
Hence, according to \eqref{12345}, we obtain, by letting $R$ be sufficiently small,
\begin{equation}\label{zz1}
\max_{t\in[0,T]}\max\left\{\|{\mathcal S}(v)(t)-1\|_{W^{1,q}(\R^3)}, \|\sigma(t)\|_{L^q(\R^3)}\right\}<\sqrt{R}.
\end{equation}
The proof of Lemma \ref{ue} is complete.
\end{proof}

\begin{Lemma}\label{ue2}
For each $1\le l\le 3$, $\f{\partial E}{\partial x_l}$ satisfies
\begin{equation}\label{ue3}
\begin{split}
\partial_t\f{\partial E}{\partial x_l}+\u\cdot\nabla\f{\partial E}{\partial x_l}=-\f{\partial \u}{\partial x_l}\cdot\nabla E+
\nabla\left(\f{\partial \u}{\partial x_l}\right)E+\nabla \u\f{\partial E}{\partial x_l}+\nabla\f{\partial \u}{\partial x_l}
\end{split}
\end{equation}
in the sense of distributions, 
that is, for all $\psi\in C^\infty_0(Q_T)$, we have
\begin{equation*}
\begin{split}
&\int_0^T\int_{\R^3}\f{\partial E}{\partial x_l}\partial_t\psi dxdt+\int_0^T\int_{\R^3}\Dv(\u\psi)\f{\partial E}{\partial x_l}\\
&=-\int_0^T\int_{\R^3}\left(-\f{\partial \u}{\partial x_l}\cdot\nabla E
+ \nabla\left(\f{\partial \u}{\partial x_l}\right)E
+\nabla \u\f{\partial E}{\partial x_l}+\nabla\f{\partial \u}{\partial x_l}\right)\psi dxdt,
\end{split}
\end{equation*}
for any $T\in(0, T_{\rm{max}})$.
\end{Lemma}

\begin{proof}
The proof is a direct application of the regularization. Indeed, one easily obtains, using \eqref{e13},
\begin{equation}\label{ue4}
\begin{split}
\partial_t (S_\varepsilon E)+\u\cdot\nabla(S_\varepsilon E)&=S_\varepsilon(\partial_t E+\u\cdot\nabla E)+\u\cdot\nabla(S_\varepsilon E)
-S_\varepsilon(\u\cdot\nabla E)\\
&=S_\varepsilon(\nabla \u E+\nabla \u)+\u\cdot\nabla(S_\varepsilon E)
-S_\varepsilon(\u\cdot\nabla E).
\end{split}
\end{equation}
Differentiate \eqref{ue4} with respect to $x_l$, we get
\begin{equation}\label{ue5}
\begin{split}
&\partial_t \left(\f{\partial S_\varepsilon E}{\partial
x_l}\right)+\u\cdot\nabla\left(\f{\partial S_\varepsilon
E}{\partial x_l}\right)\\&=
S_\varepsilon\left(\f{\partial}{\partial x_l}(\nabla \u E+\nabla
\u)\right)+\f{\partial}{\partial
x_l}\Big(\u\cdot\nabla(S_\varepsilon E)
-S_\varepsilon(\u\cdot\nabla E)\Big)-\f{\partial \u}{\partial
x_l}\cdot\nabla S_\varepsilon E.
\end{split}
\end{equation}
Notice that
\begin{equation*}
\begin{split}
\f{\partial}{\partial x_l}\Big(\u\cdot\nabla(S_\varepsilon E)
-S_\varepsilon(\u\cdot\nabla E)\Big)&=\f{\partial \u}{\partial
x_l}\cdot\nabla S_\varepsilon E-S_\varepsilon
\left(\f{\partial \u}{\partial x_l}\cdot\nabla E\right)\\
&\quad+\u\cdot\nabla S_\varepsilon\left(\f{\partial E}{\partial x_l}\right)-S_\varepsilon\left(\u\cdot\nabla \f{\partial E}{\partial x_l}\right).
\end{split}
\end{equation*}
According to Lemma 6.7 in \cite{AI} (cf. Lemma 2.3 in \cite{LI}),
we know that
$$\f{\partial \u}{\partial x_l}\cdot\nabla
S_\varepsilon E -S_\varepsilon \left(\f{\partial \u}{\partial
x_l}\cdot\nabla E\right)\to   0,\quad
\u\cdot\nabla S_\varepsilon\left(\f{\partial E} {\partial
x_l}\right)-S_\varepsilon\left(\u\cdot\nabla \f{\partial
E}{\partial x_l}\right)\to   0,$$ in $L^1(0,T; L^q(\R^3))$ as
$\varepsilon\to   0$. Hence, $$\f{\partial}{\partial
x_l}\Big(\u\cdot\nabla(S_\varepsilon E)
-S_\varepsilon(\u\cdot\nabla E)\Big)\to   0$$ in $L^1(0,T; L^q(\R^3))$. Thus, letting $\varepsilon\to   0$ in
\eqref{ue5}, we deduce
\begin{equation*}
\begin{split}
\partial_t\f{\partial E}{\partial x_l}+\u\cdot\nabla\f{\partial E}{\partial x_l}=-\f{\partial \u}{\partial x_l}\cdot\nabla E+
\nabla\left(\f{\partial \u}{\partial x_l}\right)E+\nabla \u\f{\partial E}{\partial x_l}+\nabla\f{\partial \u}{\partial x_l},
\end{split}
\end{equation*}
in the sense of weak solutions. The proof of Lemma \ref{ue2} is complete.
\end{proof}

Using \eqref{ue3}, formally we have,
\begin{equation}\label{ue6}
\begin{split}
&\int_{\R^3}\partial_t \left(\f{\partial E}{\partial
x_l}\right)\left|\f{\partial E}{\partial x_l}\right|^{q-2}\f{\partial E} {\partial x_l}dx\\
&=\int_{\R^3}\left(-\u\cdot\nabla\f{\partial E}{\partial
x_l}-\f{\partial \u}{\partial x_l}\cdot\nabla E+
\nabla\left(\f{\partial \u}{\partial x_l}\right)E+\nabla
\u\f{\partial E}{\partial x_l}+\nabla\f{\partial \u}{\partial
x_l}\right) \left|\f{\partial E}{\partial
x_l}\right|^{q-2}\f{\partial E} {\partial x_l}dx\\
&\le C\Big(\|\nabla \u\|_{L^\infty}\|\nabla
E\|^q_{L^q(\R^3)}+\|E\|_{L^\infty(Q_T)}\|\u\|_{W^{2,q}}
\|\nabla E\|^{q-1}_{L^q(\R^3)}+\|\u\|_{W^{2,q}}\|\nabla E\|_{L^q}^{q-1}\Big)\\
&\le C\Big(\|\nabla \u\|_{L^\infty}\|\nabla
\u\|^q_{L^q(\R^3)}+\sqrt{R}\|\u\|_{W^{2,q}} \|\nabla
E\|^{q-1}_{L^q(\R^3)}+\|\u\|_{W^{2,q}}\|\nabla
E\|_{L^q}^{q-1}\Big).
\end{split}
\end{equation}
We remark that the rigorous argument for the above estimate involves a tedious regularization procedure as in
DiPerna-Lions \cite{DL}, thus we omit the details and  refer the   reader to \cite{DL}.
Using \eqref{ue6}, one obtains
\begin{equation}\label{ue7}
\begin{split}
&\left\|\f{\partial E}{\partial x_l}(t)\right\|_{L^q}^p \\
&=\left\|\f{\partial E(0)}{\partial x_l}\right\|_{L^q}^p+\int_0^t\f{d}{ds}
\left\|\f{\partial E}{\partial x_l}(s)\right\|_{L^q}^p ds\\
&=\left\|\f{\partial E(0)}{\partial
x_l}\right\|_{L^q}^p+\f{p}{q}\int_0^t\left[\left\|\f{\partial
E}{\partial x_l}\right\|_{L^q}^{p-q} \int_{\R^3} \left|\f{\partial
E}{\partial x_l}\right|^{q-2}
\left(\f{\partial E}{\partial x_l}\right)\partial_s\left(\f{\partial E}{\partial x_l}(s)\right)dx\right]ds\\
&\le\left\|\f{\partial E(0)}{\partial
x_l}\right\|_{L^q}^p+C\left(\f{p}{q}\right)\int_0^t\left\|\f{\partial
E}{\partial x_l}\right\|_{L^q}^{p-q} \Big[\|\nabla
\u\|_{L^\infty}\|\nabla
E\|^q_{L^q(\R^3)}\\&\quad+\sqrt{R}\|\u\|_{W^{2,q}}
\|\nabla E\|^{q-1}_{L^q(\R^3)}+\|\u\|_{W^{2,q}}\|\nabla E\|_{L^q}^{q-1}\Big]ds\\
&\le\left\|\f{\partial E(0)}{\partial x_l}\right\|_{L^q}^p+C\left(\f{p}{q}\right)\int_0^t\|\nabla E\|_{L^q}^{p-1}
[\|\nabla \u\|_{L^\infty}\|\nabla E\|_{L^q(\R^3)}+(1+\sqrt{R})\|\u\|_{W^{2,q}}]ds\\
&\le \dl^{2p}+C\left(\f{p}{q}\right)\left(\int_0^t\|\nabla E\|_{L^q}^pds\right)^{\f{p-1}{p}}
\left(\int_0^t\|\u\|_{W^{2,q}}^pds\right)^{\f{1}{p}}\left(\max_{t\in[0,T]}\|\nabla E\|+\sqrt{R}+1\right)\\
&\le \dl^{2p}+C\left(\f{p}{q}\right)R^p\left(\max_{t\in[0,T]}\|\nabla E\|+\sqrt{R}+1\right).
\end{split}
\end{equation}
Taking the summation over $l$ in \eqref{ue7} and taking the maximum over the time $t$, one has,
\begin{equation*}
\begin{split}
\max_{t\in[0,T]}\|\nabla E\|^p\le \dl^{2p}+C\left(\f{p}{q}\right)R^p\left(\max_{t\in[0,T]}\|\nabla E\|+\sqrt{R}+1\right),
\end{split}
\end{equation*}
and hence, by letting $R, \dl$ be sufficiently small and using \eqref{2222}, we obtain,
\begin{equation}\label{ue8}
\max_{t\in[0,T]}\|\nabla E\|^p\le \dl^{2p}+CR^{p}<(\sqrt{R})^p.
\end{equation}

We are now left to deal with the quantity $\|E\|_{L^q(\R^3)}$. To this end, from the Gagliardo-Nirenberg inequality, we have
$$\|E\|_{L^q(\R^3)}\le C\|E\|_{L^2(\R^3)}^\theta
\|\nabla E\|_{L^q(\R^3)}^{1-\theta},$$
with $\theta=\f{2q}{5q-6}$. Thus, by H\"{o}lder's inequality, \eqref{w7}, and \eqref{2222}
\begin{equation}\label{GN2}
\begin{split}
\|E\|_{L^{\f{(5q-6)p}{3q-6}}(0,T;L^q(\R^3))}\le C\|E\|_{L^\infty(0,T; L^2(\R^3))}^\theta
\|\nabla E\|_{L^p(0,T; L^q(\R^3))}^{1-\theta}\le CR.
\end{split}
\end{equation}
Hence, we have the following estimate: 

\begin{Lemma}\label{lema1}
Under the same assumptions as Theorem \ref{T1}, it holds
\begin{equation}\label{EEE}
\|E\|_{L^\infty(0,T; L^q(\R^3))}<\sqrt{R},
\end{equation}
for any $T\in[0, T_{\rm{max}}]$.
\end{Lemma}

\begin{proof}
By  \eqref{e13}, \eqref{2222} and \eqref{ue8}, and letting
$$\alpha=\f{(5q-6)(p-1)}{3q-6}+1,$$  one obtains,
\begin{equation}\label{lema2}
\begin{split}
&\|E(t)\|_{L^q}^\alpha \\
&=\|E(0)\|_{L^q}^\alpha+\int_{0}^t\f{d}{ds}\|E(s)\|_{L^q}^\alpha ds\\
&=\|E(0)\|^\alpha_{L^q}+\f{\alpha}{q}\int_{0}^t\left(\|E(s)\|_{L^q}^{\alpha-q}\int_{\R^3} |E(s)|^{q-2}
E(s)\partial_sE(s)dx\right)ds\\
&=\|E(0)\|^\alpha_{L^q}+\f{\alpha}{q}\int_{0}^t\left(\|E(s)\|_{L^q}^{\alpha-q}\int_{\R^3} |E(s)|^{q-2}
E(s)\Big[\nabla \u E+\nabla \u-\u\cdot\nabla E\Big]dx\right)ds\\
&\le
\|E(0)\|^\alpha_{L^q}+\f{\alpha}{q}\int_{0}^t\left(\|E(s)\|_{L^q}^{\alpha-1}\Big[2\|\nabla
\u\|_{L^\infty}\|E\|_{L^q}+
\|\nabla \u\|_{L^q}\Big]\right)ds\\
&\le \|E(0)\|^\alpha_{L^q}+\f{\alpha}{q}\left(\int_{0}^t\|E(s)\|_{L^q}^{\f{(5q-6)p}{3q-6}}dt\right)^{\f{p-1}{p}}
\|\u\|_{L^p(0,T; W^{2,q}(\R^3))}\\
&\qquad\qquad\qquad\qquad\qquad\qquad\qquad\qquad\qquad\qquad
 \times\left(2\sup_{t\in(0,T_{\max})}\|E(t)\|_{L^q(\R^3)}+1\right)\\
&\le \|E(0)\|^\alpha_{L^q}+\left(\f{2\alpha}{q}\right)R\left(\int_{0}^t\|E(s)\|_{L^q}^{\f{(5q-6)p}{3q-6}}dt\right)^{\f{p-1}{p}}\\
&\le\|E(0)\|^\alpha_{L^q}+CR\|E\|_{L^{\f{(5q-6)p}{3q-6}}(0,T; L^q(\R^3))}^{\alpha-1}.
\end{split}
\end{equation}
Then, according to \eqref{GN2}, one has, for all $t\in[0,T_{\textrm{max}}]$,
\begin{equation}\label{GN3}
\begin{split}
\|E(t)\|_{L^q}^\alpha
\le \dl^{2\alpha}+CR^{\alpha}< \sqrt{R}^{\alpha},
\end{split}
\end{equation}
if $R$ is sufficiently small.  Thus, \eqref{EEE} follows from \eqref{GN3}.
The proof of Lemma \ref{lema1} is complete.
\end{proof}

Lemma \ref{lema1}, together with \eqref{ue8} and Lemma \ref{ue}, gives
\begin{equation}\label{ue9}
\max_{t\in[0,T]}\max\big\{\|\S(\u)-1\|_{W^{1,q}}(t),
\|\sigma\|_{L^{q}}(t), \|T(\u)\|_{W^{1,q}}(t)\big\}\le CR<\sqrt{R}.
\end{equation}
Similarly, we can obtain
\begin{equation}\label{ue99}
\max_{t\in[0,T]}\max\big\{\|\S(\u)-1\|_{W^{1,2}}(t),
\|\sigma\|_{L^2}(t), \|T(\u)\|_{W^{1,2}}(t)\big\}\le CR<\sqrt{R}.
\end{equation}

\subsection{Refined estimates on $\nabla\rho$ and $\nabla E$}

In order to prove Theorem \ref{T1}, we need some refined estimates
on $\|\nabla\r\|_{L^2(0,T; L^q(\R^3))}$ and $\|\nabla
E\|_{L^2(0,T; L^q(\R^3))}$.

\begin{Lemma} \label{rrr}
\begin{equation}\label{cl5}
\|\nabla\r\|_{L^2(0,T; L^q(\R^3))}\le CR,
\end{equation}
for any $T\in(0, T_{\rm{max}})$.
\end{Lemma}

\begin{proof}
Taking the divergence in \eqref{e12}, one obtains
\begin{equation}\label{cl3}
\D P(\r)=\Dv(\Dv(\r EE^\top))+\Dv(\Dv(\r
E))-\Dv(\r\u\cdot\nabla\u)-\Dv(\r\partial_t\u)+(\lambda+2\mu)\D\Dv\u.
\end{equation}
Since, $\Dv(\r(I+E)^\top)=0$, we get
\begin{equation*}
\Dv(\Dv(\r E))=\f{\partial}{\partial x_i}\f{\partial}{\partial
x_j}(\r E_{ij})=\f{\partial}{\partial x_j}\f{\partial}{\partial
x_i}(\r E_{ij})=-\D\r
\end{equation*}
in the sense of distributions. Hence, \eqref{cl3} becomes
\begin{equation}\label{cl4}
\D P(\r)+\D\r=\Dv\Dv(\r
EE^\top)-\Dv(\r\u\cdot\nabla\u)-\Dv(\r\partial_t\u)+(\lambda+2\mu)\D\Dv\u.
\end{equation}
Hence, one obtains, using $L^q$ theory of elliptic equations and Taylor's formula,
\begin{equation*}
\begin{split}
\|\nabla \r\|_{L^2(0,T; L^{q})}
&\le C\Big(\|\r\u\nabla\u\|_{L^2(0,T; L^{q})}+\|\Dv(\r EE^\top)\|_{L^2(0,T; L^{q})}\\
&\quad+\|\r\partial_t\u\|_{L^2(0,T; L^q)}+\|\nabla\u\|_{L^2(0,T; L^q)}\Big)\\
&\le C\Big(\|\r\|_{L^\infty(Q_T)}\|\nabla\u\|_{L^2(0,T; L^\infty)}\|\u\|_{L^\infty(0,T;
L^q)}
+\|\nabla\r\|_{L^2(0,T;
L^q)}\|E\|_{L^\infty(Q_T)}^2\\&\quad+\|\r\|_{L^\infty(Q_T)}\|\nabla
E\|_{L^2(0,T; L^q)}\|E\|_{L^\infty(Q_T)}\\&\quad+\|\r\|_{L^\infty(Q_T)}\|\partial_t\u\|_{L^2(0,T; L^q)}
+\|\nabla\u\|_{L^2(0,T; L^q)}\Big)\\
&\le C\Big(\|\r\|_{L^\infty(Q_T)}\|\u\|_{L^2(0,T;
W^{2,q})}\|\u\|_{L^\infty(0,T; L^q)}\\&\quad+\|\nabla\r\|_{L^2(0,T;
L^q)}\|E\|_{L^\infty(0,T; W^{1,q}
)}^2\\&\quad+\|\r\|_{L^\infty(Q_T)}\|\nabla E\|_{L^2(0,T;
L^q)}\|E\|_{L^\infty(0,T; W^{1,q}))}+R\Big)\\
&\le C(R^2+R)\le CR.
\end{split}
\end{equation*}
The proof of Lemma \ref{rrr} is complete.
\end{proof}

Ir order to refine $\|\nabla E\|_{L^2(0,T; L^q(\R^3))}$, we need
the following estimate:

\begin{Lemma}\label{EEE3}
\begin{equation}\label{809}
\|\partial_t\u\|_{L^2(0,T; L^q(\R^3))}\le CR^{\f{3-\t}{2}},
\end{equation}
for any $T\in(0, T_{\rm{max}})$ and some $\t\in (\f{1}{2},1]$.
\end{Lemma}

\begin{proof}
We first notice that, by the Gagliardo-Nirenberg inequality,
for $q\in (3,6]$,
\begin{equation}\label{1616}
\|\u\|_{L^2(0,T; L^q(\R^3))}\le C\|\nabla\u\|_{L^2(0,T;
L^2(\R^3))}^\theta\|\u\|_{L^2(0,T; L^2(\R^3))}^{1-\theta}\le
CR^{1+\theta},
\end{equation}
with $\theta=\f{3(q-2)}{2q}\in(\f{1}{2},1]$. Next, we multiply
\eqref{e12} by $\partial_t\u$ and integrate over $\R^3\times(0,t)$
to deduce
\begin{equation}\label{803}
\begin{split}
&\int_0^t\int_{\R^3}\r|\partial_t\u|^2 dxds
   +\int_{\R^3}(\mu|\nabla\u(t)|^2+(\lambda+\mu)|\Dv\u(t)|^2)dx\\
&=\int_{\R^3}(\mu|\nabla\u_0|^2+(\lambda+\mu)|\Dv\u_0 |^2)dx
 -\int_0^t\int_{\R^3}\r\u\cdot\nabla\u\cdot\partial_t\u dxds
 -\int_0^t\int_{\R^3}\nabla P\partial_t\u dxds \\
&\quad +\int_0^t\int_{\R^3}\Dv(\r EE^\top)\partial_t\u dxds
 +\int_0^t\int_{\R^3}\Dv(\r E)\partial_t\u dxds\\
&:=\int_{\R^3}(\mu|\nabla\u_0|^2+(\lambda+\mu)|\Dv\u_0
|^2)dx+\sum_{i=1}^4 I_i,
\end{split}
\end{equation}
with the following estimates on $I_i$ $(i=1...4)$:
recalling $Q_T=\R^3\times(0,T)$,
\begin{equation*}
\begin{split}
|I_1|&\le\|\sqrt{\r}\partial_t\u\|_{L^2(Q_T)}\|\sqrt{\r}\|_{L^\infty(Q_T)}\|\u\|_{L^\infty(0,T;
L^2(\R^3))}\|\nabla \u\|_{L^2(0,T;
L^\infty(\R^3))}\\&\le\f{1}{8}\int_0^T\int_{\R^3}\r|\partial_t\u|^2dxds+CR^4;
\end{split}
\end{equation*}
\begin{equation*}
\begin{split}
|I_2|&\le
C\|\nabla\r\|_{L^2(Q_T)}\|\sqrt{\r}\partial_t\u\|_{L^2(Q_T)}\le\f{1}{8}\int_0^T\int_{\R^3}\r|\partial_t\u|^2dxds+C\|\nabla\r\|_{L^2(Q_T)}^2\\&
\le\f{1}{8}\int_0^T\int_{\R^3}\r|\partial_t\u|^2dxds+CR^2;
\end{split}
\end{equation*}
\begin{equation*}
\begin{split}
|I_3|&\le
C\|\nabla\r\|_{L^2(Q_T)}\|E\|^2_{L^\infty(Q_T)}\|\sqrt{\r}\partial_t\u\|_{L^2(Q_T)}
+C\|E\|_{L^\infty(Q_T)}\|\nabla
E\|_{L^2(Q_T)}\|\sqrt{\r}\partial_t\u\|_{L^2(Q_T)}\\
&\le\f{1}{8}\int_0^T\int_{\R^3}\r|\partial_t\u|^2dxds+CR^4;
\end{split}
\end{equation*}
\begin{equation*}
\begin{split}
|I_4|&\le\left|\int_0^t\int_{\R^3}\partial_t(\r E)\nabla \u
dxds\right|+\left|\int_{\R^3}\r_0 E_0\nabla
\u_0dx\right|+\left|\int_{\R^3}\r(T) E(T)\nabla \u(T)dx\right|\\
&\le (\|\r\|_{L^\infty(Q_T)}\|\partial_t
E\|_{L^2(Q_T)}+\|E\|_{L^\infty(Q_T)}\|\partial_t\r\|_{L^2(Q_T)})\|\nabla\u\|_{L^2(Q_T)}\\&\quad+CR^3+(\|\nabla\r(T)\|_{L^2(\R^3)}\|E(T)\|_{L^\infty(\R^3)}+
\|\nabla
E(T)\|_{L^2(\R^3)}\|\r(T)\|_{L^\infty(\R^3)})\|\u(T)\|_{L^2(\R^3)}\\
&\le CR^3,
\end{split}
\end{equation*}
where, for the estimate $I_4$, we used equations \eqref{e11},
\eqref{e13}, Lemma \ref{ul} and estimate \eqref{1616}. Thus,
from \eqref{803}, one obtains
\begin{equation}\label{804}
\|\partial_t\u\|_{L^2(Q_T)}\le CR,
\end{equation}
and
\begin{equation}\label{805}
\begin{split}
\|\nabla\u\|_{L^\infty(0,T; L^2(\R^3))}\le CR.
\end{split}
\end{equation}
Now, we differentiate \eqref{e12} with respect to $t$, multiply
the resulting equation by $\partial_t\u$, and integrate it over
$\R^3$ to obtain
\begin{equation}\label{806}
\begin{split}
&\f{1}{2}\f{d}{dt}\int_{\R^3}\r|\partial_t\u|^2dx+\int_{\R^3}(\mu|\nabla\partial_t\u|^2+(\lambda+\mu)|\Dv
\partial_t\u|^2)dx\\&=\f{1}{2}\int_{\R^3}\partial_t\r|\partial_t\u|^2dx-\int_{\R^3}\partial_t\r
\u\cdot\nabla\u\cdot\partial_t\u
dx\\&\quad-\int_{\R^3}\r\partial_t\u\cdot\nabla\u\cdot\partial_t\u
dx-\int_{\R^3}\r\u\cdot\nabla\partial_t\u\cdot\partial_t\u
dx-\int_{\R^3}\nabla\partial_tP\partial_t\u
dx\\&\quad-\int_{\R^3}\partial_t(\r EE^\top)\nabla\partial_t \u
dx-\int_{\R^3}\partial_t(\r E)\nabla\partial_t \u
dx\\&:=\sum_{j=1}^7 J_j,
\end{split}
\end{equation}
where using \eqref{805}, we can control $J_j$ $(j=1...7)$ as
follows:
\begin{equation*}
\begin{split}
|J_1|=\left|\int_{\R^3}\nabla\r\u|\partial_t\u|^2 dx\right|\le
\|\partial_t\u\|_{L^6}^2\|\nabla\r\|_{L^3}\|\u\|_{L^3}\le
CR^2\|\nabla\partial_t\u\|_{L^2}^2;
\end{split}
\end{equation*}
\begin{equation*}
\begin{split}
|J_2|&\le
\|\partial_t\u\|_{L^6}\|\nabla\r\|_{L^3}\|\u\|_{L^6}^2\|\nabla\u\|_{L^6}\le
R\|\nabla\partial_t\u\|_{L^2}\|\nabla\u\|_{L^2}^2\|\D
\u\|_{L^2}\\
&\le R^2\|\nabla\partial_t\u\|_{L^2}^2+R^3\|\D\u\|_{L^2}^2;
\end{split}
\end{equation*}
\begin{equation*}
\begin{split}
|J_3|&\le
\|\r\|_{L^\infty}\|\nabla\u\|_{L^{\f{3}{2}}}\|\partial_t\u\|_{L^6}^2\le\|\r\|_{L^\infty}\|\u\|_{L^2}^{\f{1}{2}}\|\nabla\u\|_{L^2}^{\f{1}{2}}\|\nabla\partial_t\u\|_{L^2}^2\le
C R^{\f{7}{4}}\|\nabla\partial_t\u\|_{L^2}^2;
\end{split}
\end{equation*}
\begin{equation*}
\begin{split}
|J_4|&\le
\|\r\|_{L^\infty}\|\u\|_{L^3}\|\nabla\partial_t\u\|_{L^2}\|\partial_t\u\|_{L^6}\le
CR\|\nabla\partial_t\u\|_{L^2}^2;
\end{split}
\end{equation*}
\begin{equation*}
\begin{split}
|J_5|&\le C\|\partial_t\r\|_{L^2}\|\nabla\partial_t\u\|_{L^2}\le
C\|\nabla\r\u\|_{L^2}\|\nabla\partial_t\u\|_{L^2}\\
&\le
C\|\nabla\r\|_{L^3}\|\u\|_{L^6}\|\nabla\partial_t\u\|_{L^2}\le
C\|\nabla\r\|_{L^3}\|\nabla\u\|_{L^2}\|\nabla\partial_t\u\|_{L^2}\\
&\le
\f{\mu}{4}\|\nabla\partial_t\u\|_{L^2}^2+CR^2\|\nabla\u\|_{L^2}^2;
\end{split}
\end{equation*}
\begin{equation*}
\begin{split}
|J_6|&
\le \|\partial_t\r\|_{L^2}\|E\|_{L^\infty}^2\|\nabla\partial_t\u\|_{L^2}+\|\r\|_{L^\infty}\|E\|_{L^\infty}\|\partial_t
E\|_{L^2}\|\nabla\partial_t\u\|_{L^2}\\
&\le R^2\|\nabla\r\|_{L^3}\|\u\|_{L^6}\|\nabla\partial_t\u\|_{L^2}+R\|\nabla\u
E-\u\cdot\nabla E+\nabla\u\|_{L^2}\|\nabla\partial_t\u\|_{L^2}\\
&\le R^2\|\nabla\r\|_{L^3}\|\nabla\u\|_{L^2}\|\nabla\partial_t\u\|_{L^2}\\
&\quad +R\left(\|\nabla\u\|_{L^3}\|E\|_{L^6}+\|\nabla
E\|_{L^3}\|\u\|_{L^6}+\|\nabla\u\|_{L^2}\right)\|\nabla\partial_t\u\|_{L^2}\\
&\le
R^6\|\nabla\u\|_{L^2}^2+\f{\mu}{4}\|\nabla\partial_t\u\|_{L^2}^2+R^2(R^2\|\nabla\u\|_{L^3}^2+\|\nabla\u\|_{L^2}^2);
\end{split}
\end{equation*}
and
\begin{equation*}
\begin{split}
|J_7|&\le \|\r\|_{L^\infty}\|\partial_t
E\|_{L^2}\|\nabla\partial_t
\u\|_{L^2}+\|E\|_{L^\infty}\|\partial_t\r\|_{L^2}\|\nabla\partial_t\u\|_{L^2}\\
&\le\f{\mu}{4}\|\nabla\partial_t\u\|_{L^2}^2+R^2\|\nabla\u\|_{L^3}^2+\|\nabla\u\|_{L^2}^2+R^2\|\nabla\r\u\|_{L^2}^2\\
&\le\f{\mu}{4}\|\nabla\partial_t\u\|_{L^2}^2+R^2\|\nabla\u\|_{L^3}^2+\|\nabla\u\|_{L^2}^2+R^4\|\nabla\u\|_{L^2}^2.
\end{split}
\end{equation*}
We remark that in the above estimates, we used several times the interpolation
inequality: $$\|f\|_{W^{2,3}(\R^3)}\le
\|f\|_{W^{2,2}(\R^3)}^\theta\|f\|_{W^{2,q}(\R^3)}^{1-\theta}$$ for
some $\theta\in (0,1)$.
These estimates and \eqref{806} imply that, for $R$ sufficiently small,
\begin{equation}\label{807}
\begin{split}
\f{1}{2}\f{d}{dt}\int_{\R^3}\r|\partial_t\u|^2
dx+\f{\mu}{8}\int_{\R^3}|\nabla\partial_t\u|^2 dx\le
R^3\|\D\u\|_{L^2}^2+C\|\nabla\u\|_{L^2}^2+R^2\|\nabla\u\|_{L^3}^2.
\end{split}
\end{equation}
Integrating \eqref{807} over $(0,T)$, we obtain that, using
\eqref{ul1},
\begin{equation}\label{808}
\|\nabla\partial_t\u\|_{L^2(Q_T)}\le CR^{\f{3}{2}}.
\end{equation}
Here we used the estimate
$$\|\r_0\partial_t\u(0)\|_{L^2}\le
C(\|\u_0\cdot\nabla\u_0\|_{L^2}+\|\D\u_0\|_{L^2}+\|\nabla\r_0\|_{L^2}+\|\nabla
E_0\|_{L^2})\le \dl^4$$ by letting $t=0$ in \eqref{e12}. Thus, by
\eqref{804}, \eqref{808} and the Gagliardo-Nirenberg inequality,
we obtain
\begin{equation*} 
\|\partial_t\u\|_{L^2(0,T; L^q(\R^3))}\le
\|\partial_t\u\|_{L^2(Q_T)}^\theta\|\nabla\partial_t\u\|_{L^2(Q_T)}^{1-\theta}\le
CR^{\f{3-\t}{2}},
\end{equation*}
for some $\theta\in (\f{1}{2},1]$.
The proof of Lemma \ref{EEE3} is complete.
\end{proof}

With \eqref{809} in hand, we can now get the estimate for
$\|\nabla E\|_{L^2(0,T; L^q(\R^3))}$.

\begin{Lemma}\label{EEE2} For the same $\t\in(\f{1}{2},1]$ as in
Lemma \ref{EEE3}, it holds
\begin{equation}\label{cl2}
\|\nabla E\|_{L^2(0,T; L^q(\R^3))}\le CR^{\f{3-\t}{2}},
\end{equation}
for any $T\in(0, T_{\rm{max}})$.
\end{Lemma}

\begin{proof}
Substituting the following two facts
$$\partial_t\Dv E=\Dv(-\u\cdot\nabla E+\nabla\u E)+\D\u,\quad \Dv(\r E)=\Dv((\r-1) E)+\Dv E,$$
into \eqref{e12}, multiplying the resulting equation by $|\Dv
E|^{q-2}\Dv E$ and integrating it over $\R^3$, we can obtain
\begin{equation}\label{810}
\begin{split}
&\f{\mu}{q}\f{d}{dt}\|\Dv E\|_{L^q}^q+\|\Dv
E\|_{L^q}^q\\&\le\left|\int_{\R^3}\r\partial_t\u|\Dv E|^{q-2}\Dv
Edx\right|+\left|\int_{\R^3}\r\u\nabla\u|\Dv E|^{q-2}\Dv E
dx\right|\\&\quad+\left|\int_{\R^3}\nabla P |\Dv E|^{q-2}\Dv E
dx\right|+\left|\int_{\R^3}\Dv(\r EE^\top)|\Dv E|^{q-2}\Dv E
dx\right|\\&\quad+\left|\int_{\R^3}\Dv((\r-1)E)|\Dv E|^{q-2}\Dv E
dx\right| \\
&\quad +\left|\int_{\R^3}\Dv(\nabla\u E-\u\cdot\nabla E)|\Dv
E|^{q-2}\Dv E dx\right|\\&:=\sum_{m=1}^6 M_m,
\end{split}
\end{equation}
where
\begin{equation*}
\begin{split}
M_1\le \|\r\|_{L^\infty}\|\partial_t\u\|_{L^q}\|\Dv
E\|_{L^q}^{q-1};
\end{split}
\end{equation*}
\begin{equation*}
\begin{split}
M_2\le \|\r\|_{L^\infty}\|\u\|_{L^q}\|\nabla\u\|_{L^\infty}\|\Dv
E\|_{L^q}^{q-1}\le R\|\u\|_{W^{2,q}(\R^3)}\|\Dv E\|_{L^q}^{q-1};
\end{split}
\end{equation*}
\begin{equation*}
\begin{split}
M_3\le C\|\nabla\r\|_{L^q}\|\Dv E\|_{L^q}^{q-1};
\end{split}
\end{equation*}
\begin{equation*}
\begin{split}
M_4&\le \|\nabla\r\|_{L^q}\|E\|_{L^\infty}^2\|\Dv
E\|_{L^q}^{q-1}+\|\r\|_{L^\infty}\|E\|_{L^\infty}\|\nabla
E\|_{L^q}\|\Dv E\|_{L^q}^{q-1}\\
&\le (R^2\|\nabla\r\|_{L^q}+R\|\nabla E\|_{L^q})\|\Dv
E\|_{L^q}^{q-1};
\end{split}
\end{equation*}
\begin{equation*}
\begin{split}
M_5&\le \|\r-1\|_{L^\infty}\|\nabla E\|_{L^q}\|\Dv
E\|_{L^q}^{q-1}+\|\nabla\r\|_{L^q}\|E\|_{L^\infty}\|\Dv
E\|_{L^q}^{q-1}\\
&\le R(\|\nabla E\|_{L^q}+\|\nabla \r\|_{L^q})\|\Dv
E\|_{L^q}^{q-1};
\end{split}
\end{equation*}
\begin{equation*}
\begin{split}
M_6&\le (\|\nabla\u\|_{L^\infty}\|\nabla
E\|_{L^q}+\|\D\u\|_{L^q}\|E\|_{L^\infty})\|\Dv E\|_{L^q}^{q-1} \le
R\|\u\|_{W^{2,q}}\|\Dv E\|_{L^q}^{q-1}.
\end{split}
\end{equation*}
 With those estimates in hand, we
multiply \eqref{810} by $|\Dv E|_{L^q}^{2-q}$ to deduce that,
using Young's inequality,
\begin{equation}\label{811}
\begin{split}
\f{\mu}{2}\f{d}{dt}\|\Dv E\|_{L^q}^2+\|\Dv E\|_{L^q}^2&\le
C\|\partial_t\u\|_{L^q}^2+R^2\|\u\|_{W^{2,q}}^2+\|\nabla\r\|_{L^q}^2+R^2\|\nabla
E\|_{L^q}^2.
\end{split}
\end{equation}
On the other hand, we still have
$$\|\textrm{curl}E\|_{L^q}^2\le \|E\|_{L^\infty}^2\|\nabla
E\|_{L^q}^2\le CR^2\|\nabla E\|_{L^q}^2.$$ Hence, substituting
this into \eqref{811}, we get
\begin{equation}\label{812}
\begin{split}
\f{\mu}{2}\f{d}{dt}\|\Dv E\|_{L^q}^2+\|\nabla E\|_{L^q}^2&\le
C\|\partial_t\u\|_{L^q}^2+R^2\|\u\|_{W^{2,q}}^2+\|\nabla\r\|_{L^q}^2
+CR^2\|\nabla E\|_{L^q}^2.
\end{split}
\end{equation}
Integrating \eqref{812} over $(0,T)$ and using the estimates
\eqref{cl5}, \eqref{809}, we obtain
$$\|\nabla E\|_{L^2(0,T; L^q(\R^3))}\le CR^{\f{3-\t}{2}}.$$
The proof of Lemma \ref{EEE2} is complete.
\end{proof}

\bigskip

\section*{Acknowledgments}

Xianpeng Hu's research was supported in part by the National Science
Foundation grant DMS-0604362 and by the Mellon Predoctoral Fellowship of the University of Pittsburgh.
Dehua Wang's research was supported in part by the National Science
Foundation under grants DMS-0604362 and DMS-0906160,
and by the Office of Naval Research under Grant N00014-07-1-0668.

\end{document}